\documentclass[10pt,a4paper]{amsart}
\usepackage{amssymb,amsmath,mathrsfs}

\usepackage{times}

\usepackage[left=3.3cm,right=3.3cm,top=4cm,bottom=3.5cm]{geometry}

\usepackage{color}

\newtheorem{theorem}{Theorem}[section]
\newtheorem{lemma}[theorem]{Lemma}
\newtheorem{proposition}[theorem]{Proposition}

\newtheorem{corollary}[theorem]{Corollary}
\newtheorem{definition}[theorem]{Definition}
\newtheorem{remark}[theorem]{Remark}
\numberwithin{equation}{section}

\newcommand{\N}{\mathbb{N}}
\newcommand{\R}{\mathbb{R}}
\newcommand{\Rn}{\mathbb{R}^{n}}
\newcommand{\RN}{\mathbb{R}^{N}}
\newcommand{\M}{\mathbb{R}^{N \times n}}
\newcommand{\dd}{\mathrm{d}}
\newcommand{\tF}{\tilde{F}}
\newcommand{\tu}{\tilde{u}}

\newcommand{\tha}{\tilde{h}}
\newcommand{\BB}{\mathbb{B}}

\newcommand{\Leb}{\mathscr{L}^{n}}

\newcommand{\lip}{\operatorname{lip}}
\newcommand{\CC}{\operatorname{C}}

\newcommand{\LL}{\operatorname{L}}

\newcommand{\WW}{\operatorname{W}}
\newcommand{\BV}{\operatorname{BV}}
\newcommand{\di}{\operatorname{div}}
\newcommand{\wstar}{\stackrel{*}{\rightharpoonup}}
\newcommand{\toY}{\overset{\mathrm{Y}}{\to}}

\newcommand{\bv}{\operatorname{BV}}
\newcommand{\dif}{\operatorname{d}\!}

\newcommand{\D}{D}
\newcommand{\locc}{\operatorname{loc}}
\newcommand{\F}{\mathscr{F}}

\def\Xint#1{\mathchoice
   {\XXint\displaystyle\textstyle{#1}}%
   {\XXint\textstyle\scriptstyle{#1}}%
   {\XXint\scriptstyle\scriptscriptstyle{#1}}%
   {\XXint\scriptscriptstyle\scriptscriptstyle{#1}}%
   \!\int}
\def\XXint#1#2#3{{\setbox0=\hbox{$#1{#2#3}{\int}$}
     \vcenter{\hbox{$#2#3$}}\kern-.5\wd0}}

\def\dashint{\Xint-}

\usepackage{color}
\usepackage[citecolor=blue,colorlinks=true]{hyperref}

\begin{document}
\title{Partial regularity for BV minimizers}
\thanks{Version 22/3/2018}
\author[F.~Gmeineder]{Franz Gmeineder}
\author[J.~Kristensen]{Jan Kristensen}

\maketitle
\begin{abstract}
We establish an $\varepsilon$-regularity result for the derivative of a map of bounded variation that minimizes a strongly
quasiconvex variational integral of linear growth, and, as a consequence, the partial regularity of such $\BV$ minimizers.
This result extends the regularity theory for minimizers of quasiconvex integrals on Sobolev spaces to the context of 
maps of bounded variation. Previous partial regularity results for $\BV$ minimizers in the linear growth set-up were confined
to the convex situation.
\end{abstract}

\section{Introduction}
In this paper we investigate the local regularity properties of minimizers for variational integrals defined on Dirichlet classes 
of maps of bounded variation. In order to describe more precisely our set-up and why it is natural we consider
a continuous real-valued function defined on $N \times n$ matrices, $F \colon \M \to \R$, that we henceforth call an integrand. 
Assume that $F$ is of linear growth, that is, for some constant $L>0$ we have  
\begin{equation}\label{1grow}
|F(z)| \leq L \bigl( |z|+1 \bigr)
\end{equation}
for all matrices $z \in \M$. The reader is referred to Section \ref{sec:prelims} for undefined notation and terminology.
For a bounded Lipschitz domain $\Omega$ in $\Rn$ and a given $\WW^{1,1}=\WW^{1,1}( \Omega , \RN )$ Sobolev map 
$g \colon \Omega \to \RN$ as boundary datum we seek to minimize 
\begin{align}\label{intro1}
\int_{\Omega} \! F(\nabla v(x)) \, \dd x
\end{align}
over $v \in \WW^{1,1}_{g}=\WW^{1,1}_{g}( \Omega , \RN )$, the $\WW^{1,1}$ Dirichlet class determined by $g$. Here $\nabla v$ denotes the 
approximate Jacobi matrix that we recall coincides with the distributional derivative $Dv = \nabla v \Leb \lfloor \Omega$ when 
$v$ is a $\WW^{1,1}$ Sobolev map, thus $\nabla v(x) := \bigl[ \partial v_{j}/ \partial x_{i} (x) \bigr]$, where $j$ is 
the row number and $i$ is the column number whereby $\nabla v$ is $\M$-valued. 
The standard approach to the variational problem (\ref{intro1}) is to let the functional set-up be dictated by the coercivity inherent 
to the problem. Under the linear growth hypothesis (\ref{1grow}) the best we can hope for is that \emph{all} minimizing sequences for
(\ref{intro1}) on $\WW^{1,1}_{g}$ are bounded in the Sobolev space $\WW^{1,1}$. Building on \cite{CK} we show in Proposition 
\ref{sharpening} below that this is equivalent to the existence of constants $c_{1}>0$, $c_2 \in \R$ such that
\begin{equation}\label{mean}
\int_{\Omega} \! F(\nabla v(x)) \, \dd x \geq \int_{\Omega} \biggl( c_{1}| \nabla v(x)| + c_{2} \biggr) \, \dd x
\end{equation}
holds for all $v \in \WW^{1,1}_{g}$. We express (\ref{mean}) by saying that $F$ is mean coercive. In turn, Proposition \ref{sharpening} also
establishes the equivalence between mean coercivity and the existence of a constant $\ell > 0$ such that $F-\ell E$ is \emph{quasiconvex} 
at some $z_{0} \in \M$. Here $E \colon \M \to \R$ is our \emph{reference integrand} defined as
\begin{equation}\label{defE}
E(z) = \sqrt{1+|z|^{2}}-1.
\end{equation}
It is of course a multi-dimensional generalization from $n \geq 2$, $N=1$ of the area integrand (and of the curve length integrand when 
$n=1$, $N \geq 1$). By quasiconvexity we mean the notion introduced by \textsc{Morrey} in \cite{Morrey}, its definition is recalled in
Section 2 below. We emphasize that for $n=1$ or $N=1$, quasiconvexity is just ordinary convexity, whereas in the multi-dimensional 
vectorial case $n$, $N >1$, considered in this paper, there exists many nonconvex quasiconvex integrands of linear growth.

The question of existence of minimizers can then be successfully tackled if we assume that (\ref{mean}) holds and allow maps of bounded 
variation as minimizers. Indeed, a minimizing sequence $( u_j )$ for the problem (\ref{intro1}) is then bounded in $\WW^{1,1}$ and so admits 
a subsequence $( u_{j_k})$ so that for some $u \in \BV = \BV (\Omega , \RN )$ we have $u_{j_k} \to u$ in $\LL^1$ and of course still
$\sup_{k} \int_{\Omega} \! | \nabla u_{j_k}| \, \dd x < \infty$. We express this by writing $u_{j_k} \wstar u$ in $\BV$ and
recall that $u \colon \Omega \to \RN$ is of bounded variation, written $u \in \BV$, if it is $\LL^1$ and its distributional 
partial derivatives are measures: $Du = [ \partial u_{i}/\partial x_{j} ]$ is a bounded $\M$-valued Radon measure on $\Omega$.
We must extend the functional (\ref{intro1}) to such $u$ in a meaningful way, which in the present context is most conveniently 
done by semicontinuity following a procedure used by \textsc{Lebesgue}, \textsc{Serrin} and for quasiconvex integrals of anisotropic growth  
\textsc{Marcellini} \cite{Marcel}:
\begin{align}\label{intro3}
\F [u,\Omega ]:=\inf\left\{\liminf_{j \to \infty} \int_{\Omega} \! F(\nabla u_{j}) \, \dd x\colon\; (u_{j})\subset \WW^{1,1}_{g} (\Omega , \RN ),
\;u_{j}\to u\; \text{in} \; \LL^{1}(\Omega ,\RN) \right\}
\end{align}
Building on the works by \textsc{Ambrosio \& Dal Maso} \cite{AmbrosioDalMaso} and \textsc{Fonseca \& M\"{u}ller} 
\cite{FonsecaMueller} an integral representation for the functional $\F [u, \Omega ]$ was found in \cite{KR1} under the assumptions of 
quasiconvexity, linear growth (\ref{1grow}) and mean coercivity (\ref{mean}). For such integrands we define the 
\emph{recession integrand} by
\begin{align*}
F^{\infty}(z):=\limsup_{t\nearrow\infty}\frac{F(tz)}{t},\qquad z \in\M . 
\end{align*}
Then $F^{\infty}$ is quasiconvex and positively $1$-homogeneous \cite{Mu}. Given $u\in \BV$ we can write the
Lebesgue--Radon--Nikod\'{y}m decomposition of $Du$ into its absolutely continuous and singular parts 
with respect to $\mathscr{L}^{n}$ as
$$
\D u=\D^{ac}u+\D^{s}u=\nabla u\mathscr{L}^{n}+\tfrac{\dif \D^{s}u}{\dif |\D^{s}u|}|\D^{s}u|,
$$ 
and have then  
\begin{align}\label{intro4}
\F[u, \Omega ]=\int_{\Omega} \! F(\nabla u)\, \dd x + \int_{\Omega}\! F^{\infty}\left(\frac{\dif \D^{s}u}{\dif |\D^{s}u|}\right)\dif |\D^{s}u|+
\int_{\partial\Omega} \! F^{\infty}\big( (g-u)\otimes \nu_{\Omega} \big)\dif\mathcal{H}^{n-1},
\end{align}
where $\nu_{\Omega}$ is the outward unit normal on $\partial \Omega$. 
The last term, akin to a penalization term for failure to satisfy the Dirichlet boundary condition, must be there because the trace 
operator is not weak$\mbox{}^\ast$ continuous on $\BV$. 
We shall use the shorthand 
$$
\int_{\Omega} \! F(Du) := \int_{\Omega} \! F(\nabla u) \, \dd x + \int_{\Omega} \! F^{\infty}\left(\frac{\dif \D^{s}u}{\dif |\D^{s}u|}\right)\dif |\D^{s}u|
$$ 
for the first two terms on the right-hand side in (\ref{intro4}). 
It turns out that this expression also coincides with an extension by (area-strict) continuity of the 
integral (\ref{intro1}) initially defined on the Sobolev space $\WW^{1,1}$ (see \cite[Theorem 4]{KR1} and 
Lemmas \ref{Eapprox} and \ref{Econt} below). 
Let us summarize, under the above assumptions on $F$, we have that all minimizing sequences admit a weakly$\mbox{}^\ast$ convergent subsequence 
whose limit $u \in \BV$ is a minimizer for the functional defined at (\ref{intro4}): $\F [u,\Omega ] \leq \F [v,\Omega ]$ holds
for all $v \in \BV$. 
In particular we have for $v \in \BV$ so $u-v$ has compact support in $\Omega$ that
\begin{equation}\label{intro5}
\int_{\Omega} \! F(Du) \leq \int_{\Omega} \! F(Dv) 
\end{equation}
holds. It is clear that we should not expect that the minimality condition (\ref{intro5}) under the above assumptions on $F$ 
would entail regularity of $u$ on the Schauder $\CC^{k, \alpha}$ scale for a $k \geq 1$. For that we must evidently impose on $F$ 
a stronger quasiconvexity condition, one that in particular ensures that $F$ cannot be affine on any open subset of matrix 
space $\M$. In view of the above discussion it is natural to require that, for some fixed positive constant $\ell > 0$,
$F-\ell E$ is quasiconvex at \emph{all} $z \in \M$. That this turns out to be sufficient for regularity is our main result:

\begin{theorem}\label{thm:main}
Let $\ell$, $L > 0$ be positive constants and
suppose the integrand $F \colon \M \to \R$ satisfies the following three hypotheses:
\begin{equation}\label{H}
\begin{array}{ll}
(\mathrm{H}0) \hspace{1cm} & F\;\text{ is } \CC^{2,1}_{\mathrm{loc}}\\
{} & {}\\
(\mathrm{H}1) \hspace{1cm} & |F(z)| \leq L(|z| + 1) \quad \forall \, z \in \M\\
{} & {}\\
(\mathrm{H}2) \hspace{1cm} & z \mapsto F(z)-\ell E(z) 
\text{ is quasiconvex,}
\end{array}\nonumber
\end{equation}
where $E \colon \M \to \R$ is the reference integrand defined in (\ref{defE}).
Then for each $m>0$ there exists $\varepsilon_{m}=\varepsilon_{m}(\ell /L, F^{\prime \prime})>0$ with the following property. 
If $u \in \BV (\Omega , \RN )$ is a minimizer in the sense of (\ref{intro5}), and $B=B_{r}(x_{0}) \subset \Omega$ is a ball such that
\begin{equation}\label{sma}
|(Du)_{B}| := \left| \frac{Du (B)}{\mathscr{L}^{n}(B)} \right| < m \quad \mbox{ and } \quad
\frac{1}{\mathscr{L}^{n}(B)}\int_{B} \! E\bigl( Du-(Du)_{B}\mathscr{L}^{n}\bigr) < \varepsilon_m ,
\end{equation}
then $u$ is $\CC^{2,\alpha}$ on $B_{r/2}(x_{0})$ for each $\alpha < 1$. More precisely, $u$ is $\CC^2$ on $B_{r/2}(x_{0})$ and there 
exists a constant $c=c(\alpha ,\ell /L,F^{\prime \prime} )$ such that
\begin{equation}\label{keyregest}
\sup_{\stackrel{x,y \in B_{r/2}(x_{0})}{x \neq y}} \frac{| \nabla^{2}u (x)-\nabla^{2}u(y)|^{2}}{|x-y|^{2\alpha}} \leq 
\frac{c}{r^{n+2+2\alpha}} \int_{B_{r}(x_{0})} \! E\bigl( Du-(Du)_{B_{r}(x_{0})}\mathscr{L}^{n}\bigr).
\end{equation}
In particular, it follows that the minimizer $u$ is partially 
regular, in the sense that there exists an open subset $\Omega_{u} \subset \Omega$ such that $\mathscr{L}^{n}( \Omega \setminus \Omega_{u} ) =0$ 
and $u$ is $\CC^{2,\alpha}_{\locc}$ on $\Omega_u$ for each $\alpha < 1$.
\end{theorem}
It is important to note that without the smallness condition (\ref{sma}) we do not expect the minimizer to be regular in the sense 
of (\ref{keyregest}). This is a feature of the multi-dimensional vectorial case $n$, $N \geq 2$ rather than our assumptions, at least when
$n \geq 3$, $N \geq 2$. Indeed, for dimensions $n  \geq 3$, $N \geq 2$ there exists a regular variational integrand $F \colon \M \to \R$ 
(meaning that $F$ is $\CC^\infty$ smooth, has bounded second derivative $|F^{\prime \prime}| \leq L$ and $z \mapsto F(z)-\ell |z|^{2}$ is convex) that 
admits a Lipschitz but non-$\CC^1$ minimizer, see \cite{MoSa}. In higher dimensions, the minimizers of regular variational integrals can be more 
singular, for instance, non-Lipschitz when $n \geq 3$, $N \geq 5$ and unbounded when $n \geq 5$, $N \geq 14$, see 
\cite{SverakYan1,SverakYan2}.  
When $n=2$, $N \geq 2$ it is a result due to \textsc{Morrey} \cite{MorreyB} that minimizers of regular variational integrals must be smooth, but it is
not clear precisely how big the singular set $\Omega \setminus \Omega_u$ can be when $n \geq 3$, $N \geq 2$. Higher differentiability and 
\textsc{Gehring}'s lemma (in the adapted form \cite[Proposition 5.1]{GM1})) yield for regular variational problems that its Hausdorff dimension 
is strictly smaller than $n-2$, see \cite{BeFr,Giaquinta,Giusti,GRM} for a comprehensive discussion of this and related matters.
In the quasiconvex $p$-growth case these methods do not apply. The only result at present is \cite{KM} where it is shown that the singular set 
will be uniformly porous (and so in particular of outer Minkowski dimension strictly smaller than $n$) under the additional assumption that 
the minimizer is Lipschitz.

The underlying ideas for the proof of Theorem \ref{thm:main} have many sources and it is not easy to give proper credit. 
However, the proof strategy can be traced back to at least \textsc{De Giorgi} \cite{DeGiorgi1} and \textsc{Almgren} \cite{Almgren1,Almgren2} in their 
works on minimal surfaces in the parametric context of geometric measure theory. The first to adapt their strategy to the nonparametric case 
seem to be \textsc{Giusti \& Miranda} \cite{GiustiMiranda} and \textsc{Morrey} \cite{Morrey1}, who proved partial regularity for minimizers to
regular variational problems and weak solutions to certain nonlinear elliptic systems. The key step in these proofs is to establish a 
so-called excess decay estimate, which amounts to an integral expression of H\"{o}lder continuity. This is achieved by use of the very 
robust excess decay estimates that hold for solutions to linear elliptic systems with constant coefficients. Indeed, these excess decay 
estimates are then transferred to the minimizer/weak solution by means of a linearization procedure and Caccioppoli inequalities. In the presence
of convexity/monotonicity the required Caccioppoli inequalities are derived by use of the difference-quotient method in some form. This
method cannot be applied in the quasiconvex case. The difficulty was overcome by \textsc{Evans} \cite{Evans} who adapted an argument
used by \textsc{Widman} \cite{Widman} in another context to derive Caccioppoli inequalities of the second kind. Hereby he proved partial regularity 
of minimizers under controlled quadratic growth conditions (see \cite{Giusti} and \cite{Giaquinta} for the terminology).
Shortly afterwards \textsc{Fusco \& Hutchinson} \cite{FuHu} and \textsc{Giaquinta \& Modica} \cite{GiMo} extended the result to
minimizers of variational integrals with general integrands $F=F(x,u,\nabla u)$ of controlled $p$-growth in the $\nabla u$ variable for $p \geq 2$.
This was further extended by \textsc{Acerbi \& Fusco} \cite{AcerbiFusco1} to integrands of natural $p$-growth for $p \geq 2$. A more 
direct proof of this result was subsequently obtained by \textsc{Giaquinta} \cite{Gigrowth} who also established Caccioppoli inequalities for 
minimizers in the general case $F=F(x,u,\nabla u)$ with $p$-growth for $p \geq 2$. Let us remark that the main role of the Caccioppoli
inequalities in these proofs is to provide \emph{compactness} in some suitable context dependent sense. This is clearly seen in the blow-up 
arguments used in for instance \cite{Evans,FuHu,AcerbiFusco1}, and it was noticed by \textsc{Evans \& Gariepy} \cite{EG1} that it is possible 
to extract the necessary compactness information without explicitly going through a Caccioppoli inequality.
Partial regularity in the general subquadratic case was established by
\textsc{Carozza, Fusco \& Mingione} in \cite{CarozzaFuscoMingione}, and many interesting extensions have followed since then, these include
\cite{AcerbiFusco2,AcMi,CFPdN,FoMi,FuHu1,JCC,DLSV,DGK,DuMi,Hamburger,CH,schmidtpq}. The monograph \cite{Giusti} gives a good summary of 
the situation around the mid 90s. All the above results concern the case of variational integrals that are coercive on a Sobolev space 
$\WW^{1,p}$ for some $p>1$ and do not concern the linear growth case. The only previous partial regularity results in the 
multi-dimensional vectorial case for minimizers of variational integrals of linear growth were based on a method proposed by 
\textsc{Anzellotti \& Giaquinta} in \cite{AnGi}. While this method has been adapted by \textsc{Schmidt} \cite{Schmidt1,Schmidt2} to
cover also some degenerate convex cases, the method still crucially relies on convexity, and it cannot work for quasiconvex integrands. 
Further references on various interesting aspects of existence and regularity of minimizers in the BV context with a standard 
convexity assumption include \cite{BeSc1,Bildhauer,BildhauerFuchs,GiMoSo}.

\begin{remark}
The main point of Theorem \ref{thm:main} is that the smallness condition (\ref{sma}) under the hypotheses $\mathrm{(H0)}$, 
$\mathrm{(H1)}$, $\mathrm{(H2)}$ yields $\CC^1$ regularity of the minimizer near the point $x_0$. The fact that we obtain $\CC^{2,\alpha}$ 
regularity on $B_{r/2}(x_{0})$ for all $\alpha < 1$ is a standard outcome of this type of proof. In this connection we emphasize our 
hypothesis $\mathrm{(H0)}$ that is stronger than the usual assumption of $\CC^2$ that is normally used in this context. We invite 
the reader to check that our proof also yields $\CC^2$ regularity on $B_{r/2}(x_{0})$ of minimizers under the smallness condition 
(\ref{sma}) when $\mathrm{(H0)}$ is relaxed to 
$$
(\underline{\mathrm{H}0}) \hspace{1cm} F \mbox{ is } \CC^{2,\beta}_{\mathrm{loc}}
$$ 
for some $\beta > 1-\tfrac{1}{n}$. However, the proof does seem to require a local smoothness assumption on the integrand 
that is stronger than $\CC^2$ and it is even unclear if one can relax it beyond $(\underline{\mathrm{H}0})$. 
\end{remark}
As indicated above, we prove Theorem \ref{thm:main} by adapting the linearization procedure and Caccioppoli inequalities to
the linear growth $\BV$ scenario. In doing this there are a number of difficulties that must be overcome. The main difficulty 
turns out to be the linearization procedure, where one cannot work in the natural energy space $\WW^{1,2}$ for the linear elliptic 
system that corresponds to a suitable second Taylor polynomial of the integrand. This happens already in the case of subquadratic
growth integrands on Sobolev space $\WW^{1,p}$, but the situation in the linear growth case is more severe as it by its very nature
must be degenerate at infinity. The usual ways for implementing this step do seem to require modification. Our variant consists in 
an explicit construction of a test map that upon use delivers the required estimate. We believe this approach could be a useful 
alternative also in the standard $p$-growth case, and intend to return to this and other applications in future work.
The Caccioppoli inequality of the second kind is established following the proof given by \textsc{Evans} \cite{Evans} and 
presents no problem. It is however important to emphasize that in the linear growth case these Caccioppoli inequalities do not
allow us to establish a reverse H\"{o}lder inequality for the gradient and so we cannot prove higher integrability by use of Gehring's Lemma.
Indeed, such higher integrability is ruled out by a counterexample due to \textsc{Buckley \& Koskela} \cite{BuKo}. This can also be directly seen from 
the example of the sign function on $(-1,1)$ which \emph{does satisfy a Caccioppoli inequality} and belongs to $\bv (-1,1) \setminus \WW^{1,1}(-1,1)$. 
A brief discussion of the compactness that can be inferred from a Caccioppoli inequality of the second kind is contained in Remark \ref{caccpt} below.
We refer the interested reader to \cite{FG} for more details on this, but remark here that it is for this reason that we have so far
not been able to treat the case of minimizers for the general linear growth case $F=F(x,u,\nabla u)$. 

Finally we note that the proof of Theorem \ref{thm:main} is fairly robust. However, in view of the failure of Korn's inequality in $\LL^1$, and 
its consequence, that the space of maps of \emph{bounded deformation} $\mathrm{BD}$ is strictly larger than $\BV$, the extension of our results 
to a $\mathrm{BD}$ context under natural assumptions is not immediate. The main difficulty in transferring the proofs is that $\mathrm{BD}$ maps 
do not have an obvious \emph{Fubini property} as do $\BV$ maps (see Lemma \ref{bvrestrict}). Nevertheless this obstacle can be overcome and the 
first author has extended some of the results presented here to $\mathrm{BD}$ in his DPhil thesis \cite{FG}.

\subsection{Organization of the paper}
In Section~\ref{sec:prelims} we fix notation, collect basic facts about $\bv$-functions and record various auxiliary estimates.
We mention here in particular Subsection 2.5 on quasiconvexity that, besides recalling the relevant definitions and elementary facts, 
also makes explicit the very flexible and possibly \emph{nonconvex} nature of integrands satisfying the hypotheses $\mathrm{(H0)}$, $\mathrm{(H1)}$, 
$\mathrm{(H2)}$. Section~\ref{minseq} contains the proof of Proposition \ref{sharpening} that, as mentioned above, clarifies the role of our
strong quasiconvexity assumption $\mathrm{(H2)}$. The subsequent Section~\ref{sec:main} is devoted to the proof of Theorem~\ref{thm:main}
that we have spelled out into 5 steps, each presented in a subsection. Probably the most interesting aspect of the proof is contained in Subsection
4.3 on approximation by harmonic maps, alas the linearization procedure. Finally, we end the paper by briefly indicating possible extensions
and variants of Theorem \ref{thm:main} that can be easily established by variants of the proof given in Section~\ref{sec:main}.

\subsection*{Acknowledgments}
Both authors gratefully acknowledge the hospitality and financial support of the Max-Planck-Institute for Mathematics in the 
Natural Sciences during a stay in Leipzig in Spring 2017, where large parts of this project were concluded. The first author 
moreover acknowledges financial support by the EPSRC throughout his doctoral studies at Oxford during 2013--17 and the Hausdorff Centre of 
Mathematics in Bonn for his current postdoc position.

\section{Preliminaries}\label{sec:prelims}

\subsection{Functions of measures}\label{sec:functionsofmeasures}
Here we fix the notation and recall background facts about measures. Our reference for measure theory is \cite{AFP} whose notation and 
terminology we also follow. Let $\mathbb{H}$ be a finite dimensional Hilbert space and let $\mu$ be an $\mathbb{H}$--valued Radon measure on 
the open subset $\Omega$ of $\Rn$. Its total variation measure, denoted $| \mu |$ and defined using the norm of $\mathbb{H}$, is a nonnegative 
(possibly infinite) Radon measure on $\Omega$. We say that $\mu$ is a bounded Radon measure if it has finite total variation on $\Omega$:
$| \mu | (\Omega ) < \infty$. With respect to the $n$--dimensional Lebesgue measure $\mathscr{L}^{n}$ we have the Lebesgue--Radon--Nikod\'{y}m 
decomposition of $\mu$:
$$
\mu=\frac{\dif\mu}{\dif\mathscr{L}^{n}}\mathscr{L}^{n}+\frac{\dif\mu}{\dif |\mu^{s}|}\mu^{s}.
$$
For a Borel function $f\colon \Omega \times \mathbb{H} \to \R$ satisfying for some constant $c \geq 0$ the linear growth, or $1$-growth, condition 
$|f(x,z)| \leq c\bigl( |z|+1 \bigr)$ for all $(x,z) \in \Omega \times \mathbb{H}$ we define the (upper) recession function as 
\begin{equation}\label{urecession}
f^{\infty}(x,z):=\limsup_{\stackrel{x^{\prime} \to x, z^{\prime} \to z}{t \to \infty}} \frac{f(x^{\prime},tz^{\prime})}{t} ,\qquad (x,z) \in\Omega \times \mathbb{H}. 
\end{equation}
Hereby $f^{\infty} \colon \Omega \times \mathbb{H} \to \R$ is Borel, satisfies the growth condition $|f^{\infty}(x,z)| \leq c|z|$
for all $(x,z) \in \Omega \times \mathbb{H}$ and is positively $1$-homogeneous in its second argument: $f^{\infty}(x,tz)=tf^{\infty}(x,z)$ for 
$t \geq 0$. For $\mu$, $f$ as above we define the signed Radon measure $f(\cdot, \mu )$ by
prescribing for each Borel set $A$ whose closure is compact and contained in $\Omega$ that
$$
\int_{A} \! f(\cdot , \mu ) := \int_{A} \! f\left( \cdot \, , \frac{\dif\mu}{\dif\mathscr{L}^{n}} \right) \, 
\dd  \mathscr{L}^{n} + \int_{A} \! f^{\infty} \left(\cdot \, , \frac{\dif\mu}{\dif |\mu^{s}|}\right) \, \dd |\mu^{s}| .
$$
When $\mu$ is a bounded Radon measure the above formula extends to all Borel sets $A \subseteq \Omega$ and we easily check that
it hereby defines a bounded Radon measure on $\Omega$. When, in addition to the above, $f$ is assumed continuous and the limes superior in 
(\ref{urecession}) is a limit for all $(x,z)$, then we say that $f$ admits a regular recession function. It is then easily seen that
$f^{\infty}$ must be continuous too (as a locally uniform limit of continuous functions). Note that the function $f=1_{\Omega} \otimes E$
satisfies the above conditions and admits a regular recession function, $f^{\infty}=1_{\Omega} \otimes | \cdot |$. In fact, as is easily seen,
any continuous function $f \colon \Omega \times \mathbb{H} \to \R$ satisfying the above $1$-growth condition and so $z \mapsto f(x,z)$ is 
convex, admits a regular recession function. 

We apply in particular the above notation to functions that do not depend explicitly on $x$, so $f\colon \mathbb{H} \to \R$, and in this case
we write interchangably $f(\mu )(A)$ and $\int_{A} \! f(\mu )$ for the measure. This notation is consistent in the sense that for $f= | \cdot |$, 
$f(\mu )$ is simply the total variation measure of $\mu$ and for $f=E$, $f(\mu )+\Leb$ is the total variation measure of the $\mathbb{H} \times \R$-valued 
measure $(\mu , \Leb )$.
It is well-known that these two functionals give rise to useful notions of convergence for sequences of Radon measures. For bounded
$\mathbb{H}$-valued Radon measures on $\Omega$ we say that $\mu_j \to \mu$ strictly on $\Omega$ iff $\mu_j \wstar \mu$ in $\CC_{0}( \Omega, \mathbb{H})^{\ast}$
and $| \mu_{j}|(\Omega ) \to | \mu |(\Omega )$. A slightly stronger mode of convergence is $E$-strict or area-strict convergence on $\Omega$:
$\mu_j \wstar \mu$ in $\CC_{0}( \Omega ,\mathbb{H})^{\ast}$ and $\int_{\Omega} \! E(\mu_{j}) \to \int_{\Omega} \! E(\mu )$. Any Radon measure can 
be area-strictly approximated by smooth maps using mollification and a well-known result of Reshetnyak \cite{Reshetnyak} (and \cite[Appendix]{KR1}) 
states that for a continuous function $f \colon \Omega \times \mathbb{H} \to \R$ of $1$-growth and admitting a regular recession function we have 
$$
\int_{\Omega} \! f( \cdot , \mu_{j})  \to \int_{\Omega} \! f( \cdot , \mu )
$$
whenever $\mu_j \to \mu$ area-strictly on $\Omega$. Finally we shall often use the short-hand
$$
\int_{\Omega} \! f(\mu -z) := \int_{\Omega} \! f \bigl( \mu - z \mathscr{L}^{n} \bigr)
$$
for $z \in \mathbb{H}$ and $\mathbb{H}$-valued Radon measures $\mu$.

\subsection{Mappings of bounded variation}
Our reference for maps of bounded variation is \cite{AFP} and we follow the notation and terminology used there. Here we briefly recall
a few definitions and background results.

Let $\Omega$ be a bounded, open subset of $\Rn$. We say that an integrable map $u\colon\Omega\to\RN$ has bounded variation if its distributional 
gradient can be represented by a bounded $\M$--valued Radon measure, that is, if 
\begin{align*}
|Du|(\Omega):=\sup\left\{\int_{\Omega}u\di(\varphi)\dif x\colon\;\varphi\in \CC_{c}^{1}(\Omega ,\M ),\;|\varphi|\leq 1 \right\}<\infty. 
\end{align*}
Here and in what follows, the divergence operator, $\di$, applied to $\R^{N\times n}$--valued distributions is understood to act row--wise.
The space of maps of bounded variation is denoted by $\bv (\Omega , \RN )$ and it is a Banach space under the norm $\|v\|_{\bv}:=\|v\|_{\LL^{1}}+|\D v|(\Omega)$.
We shall use freely the results from \cite{AFP} for such maps, including in particular Poincar\'{e} and Sobolev type inequalities. 

We stress that we throughout the paper consider integrable maps in terms of their precise representatives that we define as follows. 
Assume $u \in \LL^{1}_{\mathrm{loc}}( \Omega , \mathbb{H})$, where as in the previous subsection $\mathbb{H}$ denotes a finite dimensional Hilbert 
space. We say that $u$ has approximate limit $y \in \mathbb{H}$ at $x_{0} \in \Omega$, and write
$$
\mathrm{ap}\lim_{x \to x_{0}} u(x) = y
$$
provided that
$$
\lim_{r \searrow 0} \dashint_{B_{r}(x_{0})} \! |u(x)-y| \, \dd x = 0.
$$
The set $S_u$ of points in $\Omega$ where no such limit exists is the approximate discontinuity set for $u$: 
$S_{u} = \{ x \in \Omega : \, u \mbox{ has no approximate limit at } x \}$. It is an $\Leb$ negligible
Borel set and the precise representative is defined for each $x \in \Omega \setminus S_{u}$ by (a slight abuse of notation):
$$
u(x) :=  \mathrm{ap}\lim_{x^{\prime} \to x} u(x^{\prime}).
$$
Then $u \colon \Omega \setminus S_{u} \to \mathbb{H}$ is Borel measurable, and it is not so important for the developments of this paper how we define
the precise representative on the set $S_u$. Note that when $(\rho_{\varepsilon})_{\varepsilon > 0}$ is a standard smooth mollifier and 
$u \in \LL^{1}_{\mathrm{loc}}(\Omega , \mathbb{H} )$, then $u_{\varepsilon} = \rho_{\varepsilon} \ast u$ is $\CC^{\infty}$ on 
$\Omega_{\varepsilon} = \{ x \in \Omega : \, \mathrm{dist}(x,\partial \Omega ) > \varepsilon \}$ and $u_{\varepsilon}(x) \to u(x)$ as $\varepsilon \searrow 0$
for each $x \in \Omega \setminus S_u$ (as well as locally in $\LL^1$ on $\Omega$). 
When $u$ is of bounded variation the above convergence holds in a stronger sense, though not
in the $\BV$ norm defined above. Partly for this reason it is useful to consider other modes of convergence too. We say that a sequence $(u_{k})$ in 
$\BV (\Omega , \RN )$ converges to $u\in\BV (\Omega , \RN )$ in the weak*--sense if $u_{k}\to u$ strongly in $\LL^{1}(\Omega , \RN )$ and 
$D u_{k}\wstar D u$ in $\CC_{0}(\Omega ,\M )^{\ast}$ as $k\to\infty$. 
We further say that $(u_{k})$ converges to $u$ in the $\BV$ strict sense on $\Omega$ if $u_{k}\wstar u$ and 
$|D u_{k}|(\Omega)\to |D u|(\Omega)$ as $k\to\infty$. Lastly, we say that $(u_{k})$ converges to $u$ in the 
$\BV$ area--strict sense on $\Omega$ if $u_{k}\wstar u$ and 
$$
\int_{\Omega} \! E(Du_{k}) \to \int_{\Omega} \! E(Du)
$$
as $k \to \infty$. We recall that smooth maps are dense in $\BV (\Omega ,\RN )$ in the $\BV$ area--strict sense, and more precisely:

\begin{lemma}\label{Eapprox}
Let $B=B_{R}(x_{0})$ be a ball and $u \in \BV (B,\RN )$. Then there exists a sequence $(u_{j} )$ of $\CC^{\infty}$ maps $u_{j}\colon B \to \RN$,
each of Sobolev class $\WW^{1,1}(B, \RN )$, satisfying $u_{j}|_{\partial B} = u|_{\partial B}$ and so $u_{j} \to u$ $\BV$ area--strictly on $B$.
\end{lemma}

\noindent
See for instance \cite[Lemma B.2]{Bildhauer} or \cite[Lemma 1]{KR2} for a proof that works on general domains.

\begin{lemma}\label{Econt}
Let $G \colon \M \to \R$ be rank-one convex and of linear growth: $|G(z)| \leq c(|z|+1)$ for all $z \in \M$. 
If $u$, $u_{j} \in \BV (\Omega , \RN )$, where $\Omega$ is a bounded Lipschitz domain in $\Rn$, and 
$u_{j} \to u$ $\BV$ area--strictly on $\Omega$, then
$$
\int_{\Omega} \! G(Du_{j}) \to \int_{\Omega} \! G(Du) \quad \mbox{ as } \quad j \to \infty .
$$
\end{lemma}

\noindent
We refer to \cite[Theorem 4]{KR1} for a proof.

\begin{lemma}\label{bvrestrict}
For a ball $B = B_{R}(x_{0})$ let $u \in \BV (B , \RN )$. Then for $\mathscr{L}^1$ almost all radii $r \in (0,R)$ the pointwise restriction
$u|_{\partial B_{r}}$ coincides with the traces from $B_{r}$ and from $B \setminus \overline{B_{r}}$ of $u$
and is BV on $\partial B_{r}$. Furthermore, given two radii $0<r<s<R$ we can find a radius $t \in (r,s)$ such that
$u|_{\partial B_{t}}$ is as above and its total variation over $\partial B_{t}$ is bounded as
\begin{equation}\label{goodrestrict}
\int_{\partial B_{t}} \! |D_{\tau}(u|_{\partial B_{t}})| \leq \frac{c}{s-r}\int_{B_{s}\setminus \bar{B_r}} \! |Du|,
\end{equation}
where $c=c(n,N)$ is a constant and $D_{\tau}(u|_{\partial B_{t}})$ denotes the tangential derivative (see (\ref{tangent}) below).
\end{lemma}

\begin{proof}
We can assume that $x_0 = 0$. For a standard smooth mollifier $( \rho_{\varepsilon})_{\varepsilon > 0}$
we put $u_{\varepsilon} = \rho_{\varepsilon} \ast u$. Then $u_{\varepsilon} \in \CC^{\infty}(B_{R-\varepsilon} , \RN )$ and
we have $u_{\varepsilon} \to u$ $\BV$ strictly on $B_{s'}\setminus \overline{B}_{r'}$ for any radii $r \leq r' < s' \leq s$ with
$|Du|(\partial B_{s'} \cup \partial B_{r'})=0$. Furthermore, $u_{\varepsilon}(x) \to u(x)$ for each 
$x \in B \setminus S_{\bar{u}}$ as $\varepsilon \searrow 0$.

The tangential derivative of $u_{\varepsilon}$ at $x \in \partial B_{t}$ on 
the sphere $\partial B_{t}$ is given by
\begin{equation}\label{tangent}
\nabla_{\tau}(u_{\varepsilon}|_{\partial B_{t}})(x) = \nabla u_{\varepsilon}(x)\bigl( I - \tfrac{x \otimes x}{t^{2}} \bigr)
\end{equation}
so by integration in polar coordinates
\begin{eqnarray*}
\int_{r}^{s} \! \int_{\partial B_{t}} \! |D_{\tau}(u_{\varepsilon}|_{\partial B_{t}})| \, \dd t &=& 
\int_{r}^{s} \! \int_{\partial B_{t}} \! |\nabla_{\tau}(u_{\varepsilon}|_{\partial B_{t}})| \, \dd \mathcal{H}^{n-1} \, \dd t\\
&\leq& \int_{r}^{s} \! \int_{\partial B_{t}} \! |\nabla u_{\varepsilon}| \, \dd \mathcal{H}^{n-1} \, \dd t\\
&=& \int_{B_{s} \setminus B_{r}} \! | \nabla u_{\varepsilon}| \, \dd x = \int_{B_{s} \setminus B_{r}} \! |Du_{\varepsilon}|.
\end{eqnarray*}
The set $M = \{ t \in (r,s) : \, \mathcal{H}^{n-1}(S_{u} \cap \partial B_{t}) > 0 \}$ is $\mathscr{L}^1$ negligible, and for $t \in (r,s)\setminus M$
we have that, as $\varepsilon \searrow 0$,  $u_{\varepsilon}(x) \to u(x)$ for $\mathcal{H}^{n-1}$ a.e. $x \in \partial B_t$ and by the trace theorem
(see \cite{AFP} Th. 3.77) also in $\LL^{1}(\partial B_{t}, \RN )$.
Next, Fatou's lemma and the strict convergence give for the radii $r \leq r' < s' \leq s$ selected above that
$$
\int_{r'}^{s'} \! \liminf_{\varepsilon \searrow 0} \int_{\partial B_{t}} \! |\nabla_{\tau}(u_{\varepsilon}|_{\partial B_{t}})| 
\, \dd \mathcal{H}^{n-1} \, \dd t \leq \int_{B_{s}\setminus \bar{B_{r}}} \! |Du| ,
$$
and hence taking $r' \searrow r$ and $s' \nearrow s$ we get
$$
\int_{r}^{s} \! \liminf_{\varepsilon \searrow 0} \int_{\partial B_{t}} \! |\nabla_{\tau}(u_{\varepsilon}|_{\partial B_{t}})| 
\, \dd \mathcal{H}^{n-1} \, \dd t \leq \int_{B_{s}\setminus \bar{B_{r}}} \! |Du| .
$$
For each $t \in (r,s) \setminus M$ such that
$$
\liminf_{\varepsilon \searrow 0} \int_{\partial B_{t}} \! |\nabla_{\tau}(u_{\varepsilon}|_{\partial B_{t}})| 
\, \dd \mathcal{H}^{n-1} < \infty
$$
which is $\mathscr{L}^1$ almost all $t$, we can find a subsequence $\varepsilon_{j}=\varepsilon_{j}(t) \searrow 0$ such that
$u_{\varepsilon_{j}}|_{\partial B_t} \to u|_{\partial B_t}$ in $\LL^{1}(\partial B_{t} , \RN )$ and pointwise $\mathcal{H}^{n-1}$ a.e., and
$$
\lim_{j \to \infty} \int_{\partial B_{t}} \! |\nabla_{\tau}(u_{\varepsilon_{j}}|_{\partial B_{t}})| \, \dd \mathcal{H}^{n-1} < \infty .
$$
This implies that $u|_{\partial B_{t}} \in \BV (\partial B_{t}, \RN )$. Finally, the last assertion follows because we can select 
$t \in (r,s) \setminus M$ so
$$
\liminf_{\varepsilon \searrow 0} \int_{\partial B_{t}} \! |\nabla_{\tau}(u_{\varepsilon}|_{\partial B_{t}})| 
\, \dd \mathcal{H}^{n-1} \leq \tfrac{2}{s-r} \int_{B_{s}\setminus \bar{B_{r}}} \! |Du|,
$$
and then conclude by selecting a suitable subsequence $\varepsilon_{j} \searrow 0$ as above.
It follows that the pointwise restriction of the precise representative $u|_{\partial B_t} \in \BV (\partial B_{t}, \RN )$ coincides
$\mathcal{H}^{n-1}$ a.e. with the traces of $u$ from $B_{t}$ and from $B_{s}\setminus \overline{B}_{t}$ and that (\ref{goodrestrict}) holds.
\end{proof}

\noindent
For the statement of the next result we recall that for a ball $B=B_{R}(x_{0})$ in $\Rn$, $s \in (0,1)$ and $p \in (1,\infty )$ the 
Sobolev-Slobodecki\u{\i} spaces $\WW^{s,p}(B, \RN )$ and  $\WW^{s,p}(\partial B , \RN )$ consist of all integrable maps $u \colon B \to \RN$,
$v \colon \partial B \to \RN$ for which the Gagliardo norm
$$
\| u \|_{\WW^{s,p}(B , \RN )} = \left( \| u \|_{\LL^{p}(B, \RN )}^{p}+| u |_{\WW^{s,p}(B , \RN )}^{p} \right)^{\frac{1}{p}},
$$
$$
\| v \|_{\WW^{s,p}(\partial B , \RN )} = \left( \| v \|_{\LL^{p}(\partial B, \RN )}^{p}+| v |_{\WW^{s,p}(\partial B , \RN )}^{p} \right)^{\frac{1}{p}}
$$
is finite, respectively. Here we define the corresponding semi-norms as, respectively,
$$
| u |_{\WW^{s,p}(B , \RN )} = \left( \int_{B} \int_{B} \! \frac{|u(x)-u(y)|^{p}}{|x-y|^{n-1+sp}} \, \dd x \, \dd y \right)^{\frac{1}{p}}.
$$
and
$$
| v |_{\WW^{s,p}(\partial B , \RN )} = \left( \int_{\partial B} \int_{\partial B} \! \frac{|v(x)-v(y)|^{p}}{|x-y|^{n-1+sp}} \, \dd \mathcal{H}^{n-1}(x) \, 
\dd \mathcal{H}^{n-1}(y) \right)^{\frac{1}{p}}.
$$
\begin{lemma}\label{bvembedding}
Assume the dimension $n \geq 3$.
Let $B=B_{R}(x_{0})$ be a ball and $v \in \BV (\partial B,\RN )$. Then $v \in \WW^{\frac{1}{n},\frac{n}{n-1}}(\partial B , \RN )$ and
$$
\left( \dashint_{\partial B} \int_{\partial B} \! \frac{|v(x)-v(y)|^{\frac{n}{n-1}}}{|x-y|^{n-1+\frac{1}{n-1}}} \, \dd \mathcal{H}^{n-1}(x) \, 
\dd \mathcal{H}^{n-1}(y) \right)^{1-\frac{1}{n}} \leq cR^{1-\frac{1}{n}}\dashint_{\partial B} \! |D_{\tau}v| ,
$$
where $c=c(n,N)$ is a constant.
\end{lemma}
We refer to \cite[Lemma D.2]{BoBrMi} for a proof that $\BV (\R^{n-1} )$ for dimensions $n \geq 3$ embeds into $\WW^{\frac{1}{n},\frac{n}{n-1}}(\R^{n-1})$.
Lemma \ref{bvembedding} can be recovered from this result by the usual arguments involving local coordinates and a partition of unity. Note that 
the embedding fails for dimension $n=2$: an indicator function for a circular arc has bounded variation on $\partial B$ but it is not of class 
$\WW^{\frac{1}{2},2}(\partial B)$. In the two-dimensional case we instead have an embedding into the larger $\LL^2$ based Nikolski\u{\i} space that 
we in this context may define as
$$
\mathrm{B}^{\tfrac{1}{2},2}_{\infty}(\partial B , \RN ) = \left\{ v \in \LL^{2}(\partial B , \RN ) : \, \sup_{0<|h|< R/2} \int_{\partial B} \! 
| v \left( \tfrac{x+h}{|x+h|}\right)-v(x)|^{2} \, \dd \mathcal{H}^{1}(x)/|h| < \infty \right\} .
$$ 
This definition is easily seen to be equivalent to the one obtained by transferring the usual definition on the interval $(-1,1)$ by local coordinates 
and a partition of unity. A proof of the aforementioned embedding can therefore be inferred from \cite[Lemma 38.1]{Tartar}. In combination with a 
Sobolev embedding result (see \cite[Lemma 22.2, (34.4) and Lemma 36.1]{Tartar} or \cite[Theorem 4.6.1(a)]{Triebel}) we then deduce:

\begin{lemma}\label{bvembedding2}
Assume the dimension $n=2$.
Let $B=B_{R}(x_{0}) \subset \R^2$ be a ball and $v \in \BV (\partial B,\RN )$. Then $v \in \WW^{1-\frac{1}{p},p}(\partial B , \RN )$ for each $p \in (1,2)$
and
$$
\left( \dashint_{\partial B} \int_{\partial B} \! \frac{|v(x)-v(y)|^{p}}{|x-y|^{p}} \, \dd \mathcal{H}^{1}(x) \, 
\dd \mathcal{H}^{1}(y) \right)^{\frac{1}{p}} \leq cR^{\frac{1}{p}}\dashint_{\partial B} \! |D_{\tau}v| ,
$$
where $c=c(N,p)$ is a constant.
\end{lemma}

When $u \in \CC^{0}( \overline{\Omega},\RN )$ we denote by $\mathrm{Tr}_{\Omega}(u)=u|_{\partial \Omega}$ the
primitive trace operator of $u$ on $\partial \Omega$, and when we write $\mathrm{Tr}_{\Omega}(v)$ for more general Sobolev
mappings $v$ below we understand as usual this trace as the extension by continuity of the primitive trace operator to the relevant  
space.
We refer to \cite{Tartar,Triebel} for further background on Besov and Sobolev-Slobodecki\u{\i} spaces. However, for later reference we explicitly recall two 
instances of Gagliardo's trace theorem here.

\begin{lemma}\label{bvtrace}
For bounded Lipschitz domains $\Omega$ in $\Rn$ the trace operator $u \mapsto u|_{\partial \Omega}$ extends from smooth maps on $\overline{\Omega}$
by strict continuity to a well-defined strictly continuous linear surjective operator 
$$
\mathrm{Tr}_{\Omega} \colon \BV (\Omega , \RN ) \to \LL^{1}(\partial \Omega , \RN ).
$$
Furthermore, we already have $\mathrm{Tr}_{\Omega} \bigl( \WW^{1,1}(\Omega , \RN ) \bigr) = \LL^{1}(\partial \Omega , \RN )$.
In particular we have for a ball $B=B_{R}(x_{0})$, writing $\bar{u}=\mathrm{Tr}_{B}(u)$ for $u \in \BV (B, \RN )$ that
\begin{equation}\label{sharptrace}
\int_{\partial B} \! \left| \bar{u}-\dashint_{\partial B} \bar{u} \, \dd \mathcal{H}^{n-1} \right| \, \dd \mathcal{H}^{n-1} \leq c\int_{B} \! |Du|,
\end{equation}
where $c=c(n,N)$ is a constant.
\end{lemma}

\begin{lemma}\label{gagtrace}
For the unit ball $\BB=B_{1}(0)$ there exists a bounded linear extension operator
$$
\mathrm{E} \colon \WW^{k-\frac{1}{p},p}(\partial \BB , \RN ) \to \WW^{k,p}( \BB, \RN ) , \quad k \in \N , \quad 1<p< \infty .
$$
More precisely, $\mathrm{E}$ does not depend on $k$, $p$ and is a bounded linear operator such that 
$\mathrm{Tr}_{\BB} \circ \mathrm{E}$ is the identity on $\WW^{k-\frac{1}{p},p}(\partial \BB , \RN )$.
\end{lemma}

\subsection{Auxiliary estimates for the reference integrand $E$}

We write $E(z)$ for the reference integrand defined at (\ref{defE}) whenever $z \in \mathbb{H}$
and $\mathbb{H}$ is a finite dimensional Hilbert space. Firstly, elementary estimations yield
\begin{equation}\label{minb}
\left\{
\begin{array}{l}
(\sqrt{2}-1)\min \{ |z|,|z|^{2} \} \leq E(z) \leq \min \{ |z|,|z|^{2} \}\\
\mbox{ }\\
E(az) \leq a^{2}E(z) \quad \mbox{ and } \quad E(z+w) \leq 2\bigl( E(z)+E(w) \bigr)
\end{array}
\right.
\end{equation}
for all $z$, $w \in \mathbb{H}$ and $a\geq 1$. 
For the following, define for a measurable subset $A$ of $\R^{n}$ and a $\mathbb{H}$--valued Radon 
measure $\mu$ on $\R^{n}$ the mean value $\mu_{B}:=\mu(B)/\Leb (B)$. 
\begin{lemma}\label{estE1}
Let $\phi$ be a bounded $\mathbb{H}$--valued Radon measure on an open ball $B$ in $\Rn$. Then
\begin{equation}\label{qminE}
\int_{B} \! E( \phi -\phi_{B}) \leq 4\int_{B} \! E( \phi -z)
\end{equation}
for all $z \in \mathbb{H}$.
\end{lemma}

\begin{proof}
By mollification we can asssume that $\phi \in \LL^{1}(B, \mathbb{H})$. From (\ref{minb}) and convexity we find
\begin{eqnarray*}
\int_{B} \! E( \phi -\phi_{B}) \, \dd x 
&\leq& 2\int_{B} \! E(\phi -z) \, \dd x +2\Leb (B)E(\phi_{B}-z)\\
&\leq& 4\int_{B} \! E(\phi -z) \, \dd x
\end{eqnarray*}
as required. 
\end{proof}

\begin{lemma}\label{estE2}
Let $\phi$ be a bounded $\mathbb{H}$--valued Radon measure on an open ball $B$ in $\Rn$.
Then
$$
\dashint_{B} \! | \phi | \leq \sqrt{{\mathscr{E}^{2}+2\mathscr{E}}}, \quad \mbox{ where } \, \mathscr{E} = \dashint_{B} \! E(\phi ) .
$$
In particular, for $\mathscr{E} \leq 1$ we have
\begin{equation}\label{smallE}
\dashint_{B} \! | \phi | \, \dd x \leq \sqrt{3\mathscr{E}}.
\end{equation}
\end{lemma}

\begin{proof}
By mollification we can asssume that $\phi \in \LL^{1}(B, \mathbb{H})$. From Jensen's inequality
$$
E\left( \dashint_{B} \! | \phi | \, \dd x \right) \leq \dashint_{B} \! E( \phi ) \, \dd x = \mathscr{E},
$$
and hence solving for the $\LL^1$ norm we easily conclude.
\end{proof}

\subsection{Estimates for Legendre-Hadamard elliptic systems}
The space of symmetric and real bilinear forms on $\M$ is denoted by $\bigodot^{2}( \M )$ and equipped with the operator norm, denoted
and defined for $\mathbb{A} \in \bigodot^{2}( \M )$ as $| \mathbb{A}| = \sup \{ \mathbb{A}[z,w] : \, |z|, |w| \leq 1 \}$. Observe that
the precise meaning of $| \cdot |$ can be understood from the context, and for a matrix $z \in \M$ we use it to denote the usual euclidean
norm: $|z|^{2} = \mbox{trace} (z^{t}z)$. Likewise for vectors in $\mathbb{R}^k$.
Fix $\mathbb{A} \in \bigodot^{2}( \M )$ and assume it satisfies the strong Legendre-Hadamard condition
\begin{equation}\label{sLH}
\left\{
\begin{array}{l}
\alpha | y |^{2}| x |^{2} \leq \mathbb{A} [ y \otimes x , y \otimes x ] 
\quad \forall y \in \RN , \, \forall x \in \Rn\\
| \mathbb{A} | \leq \beta ,
\end{array}
\right.
\end{equation}
where $\alpha$, $\beta > 0$ are constants. Any $\RN$-valued distribution $u$ on
$\Omega$ satisfying
$$
-\mathrm{div} \, \mathbb{A}Du = 0 \quad \mbox{ in the distributional sense on } \Omega
$$
where $\mathrm{div}$ is understood to act row-wise, is called $\mathbb{A}$-harmonic, or simply harmonic when $\mathbb{A}$ is
clear from the context. 

The next lemma is a standard Weyl-type result and can for instance be proved using the difference-quotient method (see 
\cite{Giaquinta,Giusti,Mit,MorreyB}).

\begin{lemma}\label{Weyl}
Let $\mathbb{A}\in \bigodot^{2}( \M )$ satisfy (\ref{sLH}). Then there exists a constant $c=c(\frac{\beta}{\alpha},n,N)$ with
the following propeties. Let $B=B_{R}(x_{0})$ be a ball in $\Rn$ and assume that $h \in \WW^{1,1}(B, \RN )$ is harmonic in $B$: 
$-\mathrm{div} \mathbb{A}\nabla h = 0$ in $B$. Then $h$ is $\CC^{\infty}$ on $B$ and for any $z \in \M$ we have
$$
\sup_{B_{\frac{R}{2}}} | \nabla h-z| + R\sup_{B_{\frac{R}{2}}} | \nabla^{2} h| \leq c\dashint_{B_{R}} \! |\nabla h -z| \, \dd x.
$$
\end{lemma}


Finally we state two basic existence and regularity results for inhomogeneous Legendre-Hadamard elliptic systems that are 
instrumental for our arguments below. 
\begin{proposition}\label{exH}
Let $\mathbb{A}\in \bigodot^{2}( \M )$ satisfy (\ref{sLH}) and fix exponents $p \in (1,\infty )$ and $q \in [2, \infty )$. 
Denote $\BB = B_{1}(0)$, the open unit ball in $\Rn$.

(a) For each $g \in \WW^{1-\tfrac{1}{p},p}(\partial \BB ,\RN )$ there exists a unique solution $h \in \WW^{1,p}( \BB , \RN )$ to the
elliptic system
\begin{equation}\label{sys1}
\left\{
\begin{array}{ll}
-\mathrm{div } \mathbb{A}\nabla h = 0 & \mbox{ in } \BB\\
h|_{\partial \BB} = g & \mbox{ on } \partial \BB ,
\end{array}
\right.
\end{equation}
and 
$$
\| h \|_{\WW^{1,p}} \leq c \| g \|_{\WW^{1-\frac{1}{p},p}}
$$ 
where $c=c(n,N,p,\tfrac{\alpha}{\beta})$.

(b) For each $f \in \LL^{q}( \BB , \RN )$ there exists a unique solution $w \in \WW^{2,q}( \BB , \RN )$ to the
elliptic system
\begin{equation}\label{sys2}
\left\{
\begin{array}{ll}
-\mathrm{div } \mathbb{A} \nabla w = f & \mbox{ in } \BB\\
w|_{\partial \BB} = 0 & \mbox{ on } \partial \BB ,
\end{array}
\right.
\end{equation}
and 
$$
\| w \|_{\WW^{2,q}} \leq c\| f \|_{\LL^q},
$$ 
where $c=c(n,N,q,\tfrac{\alpha}{\beta})$.
\end{proposition}
While these results are well-known we have been unable to find a precise reference. They can 
be inferred from more general results stated in \cite{MorreyB}, see in particular Theorems 6.4.8 and 6.5.5 there, and also from 
\cite{MaSh}, Lemma 3.2 (taking the remark on page 106 into account). The last reference does not provide details 
for the general Legendre-Hadamard elliptic case, but the reader can find the nontrivial calculations and further background
in the book \cite{Mit}. 
All the above mentioned proofs rely on boundary layer methods, and the work on these is still ongoing
with many interesting open questions remaining, see for instance \cite{M3M}. However the proof of Proposition \ref{exH}
need not be so sophisticated. An easier route goes via the elegant approach exposed by \textsc{Giusti} in \cite[Chapter 10]{Giusti}. As stated
there it builds on earlier works by \textsc{Stampacchia} \cite{Stamp} and \textsc{Campanato} \cite{Camp}, and derives $\LL^p$ estimates from simple $\LL^2$
estimates and interpolation. For the convenience of the reader we provide a brief sketch along these lines.
\bigskip

\noindent
\textit{Sketch of Proof.}
(a): By virtue of Gagliardo's trace theorem, as stated in Lemma \ref{gagtrace}, we can find an extension 
$\bar{g} \in \WW^{1,p}( \BB , \RN )$ with $\bar{g}|_{\partial \BB}=g$ and 
\begin{equation}\label{gagsub}
\| \bar{g} \|_{\WW^{1,p}} \leq c\| g \|_{\WW^{1-\frac{1}{p},p}}
\end{equation} 
for a constant $c=c(n,N,p)$. If we put $\bar{h}=h-\bar{g}$, then by simple substitution we see that we can shift attention 
from (\ref{sys1}) to the system
\begin{equation}\label{sys1sub}
\left\{
\begin{array}{ll}
-\mathrm{div} \mathbb{A}\nabla \bar{h} = -\mathrm{div} V & \mbox{ in } \BB ,\\
\bar{h}|_{\partial \BB}=0 & \mbox{ on } \partial \BB ,
\end{array}
\right.
\end{equation}
where $V=\mathbb{A}\nabla \bar{g} \in \LL^{p}( \BB , \M )$. Now for $p \in [2, \infty )$ existence, uniqueness and $\LL^p$ estimate all
follow from \cite[Theorem 10.15]{Giusti} and (\ref{gagsub}). 

It remains to consider the subquadratic case $p \in (1,2)$. 
In this situation we take $V_j \in \LL^{2}( \BB , \M )$ so $\| V-V_j \|_{\LL^p} \to 0$, and let
$\bar{h}_j \in \WW^{1,2}_{0}( \BB , \RN )$ be the unique solution to
\begin{equation}\label{sys1sub1}
\left\{
\begin{array}{ll}
-\mathrm{div} \mathbb{A}\nabla \bar{h}_j = -\mathrm{div} V_j & \mbox{ in } \BB ,\\
\bar{h}_{j}|_{\partial \BB}=0 & \mbox{ on } \partial \BB .
\end{array}
\right.
\end{equation}
Note that $W_j = | \nabla \bar{h}_{j}|^{p-2}\nabla \bar{h}_j \in \LL^{p^{\prime}}( \BB , \M )$, where $p^{\prime} \in (2,\infty )$ is the H\"{o}lder
conjugate exponent of $p$. Consequently, we infer from the above concluded superquadratic case that the elliptic system
\begin{equation}\label{sys1sub2}
\left\{
\begin{array}{ll}
-\mathrm{div} \mathbb{A}\nabla \varphi_j = -\mathrm{div} W_j & \mbox{ in } \BB ,\\
\varphi_{j}|_{\partial \BB}=0 & \mbox{ on } \partial \BB 
\end{array}
\right.
\end{equation}
admits a unique solution $\varphi_j \in \WW^{1,p^{\prime}}_{0}( \BB , \RN )$ with
\begin{equation}\label{estfi}
\| \varphi_j \|_{\WW^{1,p^{\prime}}} \leq c \| W_j \|_{\LL^{p^{\prime}}} = c \| \nabla \bar{h}_{j} \|_{\LL^p}^{p-1}
\end{equation}
where $c=c(n,N,p^{\prime},\tfrac{\alpha}{\beta})$. If we test (\ref{sys1sub1}) with $\varphi_j$ and use that $\mathbb{A}$ is symmetric, then
$\| \nabla \bar{h}_j \|_{\LL^p} \leq c \| V_j \|_{\LL^p}$ results. Now Poincar\'{e}'s inequality and (\ref{gagsub}) easily allow us to conclude
that there exists a solution to (\ref{sys1}) satisfying the $\LL^p$ estimate. It remains to prove uniqueness in the subquadratic case. To that
end we assume $h \in \WW^{1,p}_{0}( \BB , \RN )$ satisfies $-\mathrm{div} \mathbb{A}\nabla h = 0$ in $\BB$. From Lemma \ref{Weyl} we know that
$h \in \CC^{\infty}( \BB , \RN )$ and so if we for $r \in (0,1)$ define $h_{r}(x)=h(rx)$, then clearly 
$h_{r}|_{\partial \BB} \in \WW^{\frac{1}{2},2}( \partial \BB , \RN )$ and $-\mathrm{div} \mathbb{A}\nabla h_{r}=0$ in $\BB$. By uniqueness of $\WW^{1,2}$
solutions it follows from the above that $\| h_{r} \|_{\WW^{1,p}} \leq c\| h_{r}|_{\partial \BB} \|_{\WW^{1-\frac{1}{p},p}}$. But clearly $h_r \to h$ in
$\WW^{1,p}( \BB , \RN )$ as $r \nearrow 1$, so $\| h_{r}|_{\partial \BB} \|_{\WW^{1-\frac{1}{p},p}} \to 0$ as $r \nearrow 1$ by the continuity of trace,
and thus $h=0$.

(b): Since $q \in [2, \infty )$ the assertion follows directly from \cite{Giusti}, see (10.60)--(10.63) 
on pp.~369--370. \hfill $\square$

\subsection{Quasiconvexity}

We start by displaying \textsc{Morrey}'s definition of quasiconvexity \cite{Morrey,MorreyB}:

\begin{definition}\label{defqc}
A continuous integrand $G \colon \M \to \R$ is quasiconvex at $z_{0} \in \M$ provided
$$
G(z_{0})\leq \int_{(0,1)^{n}} \! G(z_{0}+\nabla \varphi (x))\, \dd x
$$
holds for all compactly supported Lipschitz maps $\varphi \colon (0,1)^{n} \to \RN$. It is quasiconvex if it is quasiconvex at all $z_{0} \in \M$.
\end{definition}

It is well-known (see \cite{Dacorogna,MorreyB}) that quasiconvexity implies rank-one convexity, and that rank-one convexity and
linear growth, say $|G(z)| \leq L(|z|+1)$ for all $z \in \M$, yield a Lipschitz bound that for $\CC^1$ integrands takes the form
\begin{equation}\label{lip}
|G^{\prime}(z)| \leq cL \quad \forall \, z \in \M .
\end{equation}
The proof in \cite{BKK} gives (\ref{lip}) with the constant $c=\sqrt{\min \{ n,N \}}$. 

When a quasiconvex integrand $G$ has linear growth it means that the quasiconvexity inequality can be tested by more general maps.
We have from \cite[Proposition 1]{KR1}:

\begin{lemma}\label{1qc}
Assume $G \colon \M \to \R$ is quasiconvex and of linear growth. If $\omega$ is a bounded Lipschitz domain in $\Rn$, 
$\varphi \in \BV ( \omega , \RN )$ and $a \colon \Rn \to \RN$ is affine, then
$$
\Leb (\omega ) G( \nabla a) \leq \int_{\omega} \! G(D\varphi ) + \int_{\partial \omega} \! G^{\infty} 
\bigl( (a-u) \otimes \nu_{\omega} \bigr) \, \dd \mathscr{H}^{n-1} 
$$
holds, where $\nu_{\omega}$ is the outward unit normal on $\partial \omega$.
\end{lemma}
As explained in the Introduction our quasiconvexity assumption $\mathrm{(H2)}$, that we shall refer to as 
\emph{strong quasiconvexity}, is very natural when compared to the minimal set of conditions that allows one to prove 
existence of a $\BV$ minimizer by use of the direct method.
We emphasize that quasiconvexity is much more general than convexity and refer to \cite{Dacorogna} for a long list 
of examples of nonconvex quasiconvex integrands. That there exists nonconvex quasiconvex integrands of linear growth 
is also well-known (see \cite{Mu,KZ}). Here we shall briefly illustrate the abundance of nonconvex integrands $F$ 
satisfying the hypotheses $\mathrm{(H0)}$, $\mathrm{(H1)}$, $\mathrm{(H2)}$. In fact many of these integrands are 
nonconvex in the sense that also the functional
$$
v \mapsto \int_{\Omega} \! F( \nabla v) \, \dd x
$$
is nonconvex on the Dirichlet class $\WW^{1,1}_{g}$.

\begin{proposition}\label{noncvx}
Let $F \colon \M \to \R$ be quasiconvex and assume that for some exponent $p \in (1,\infty )$ and constant $L \geq 1$
we have the $p$-coercivity-growth condition:
$$
|z|^{p} \leq F(z) \leq L\bigl( |z|^{p}+1 \bigr) \quad \forall \, z \in \M .
$$
Then there exists a sequence $(F_j )$ of integrands $F_{j} \colon \M \to \R$ satisfying for
some constants $\ell_j$, $L_j > 0$,
$$
\begin{array}{ll}
(0) \hspace{1cm} & F_{j}\;\text{ is } \CC^{\infty}\\
{} & {}\\
(1) \hspace{1cm} & |z|-1 \leq F_{j}(z) \leq L_{j}(|z| + 1) \quad \forall \, z \in \M\\
{} & {}\\
(2) \hspace{1cm} & z \mapsto F(z)-\ell_{j} E(z) 
\text{ is quasiconvex,}
\end{array}
$$
such that $F_{j}(z) \nearrow F(z)$ as $j \nearrow \infty$ pointwise in $z \in \M$.
\end{proposition}
The proof is an easy adaptation of \cite[Proposition 1.10]{JK}. Observe that nonconvexity of $F$ or the 
corresponding functional must be inherited by elements of the approximating sequence $F_j$ for sufficiently 
large values of $j$. We could therefore for instance apply Proposition \ref{noncvx} to the integrands constructed 
in \cite{MuSv} to get the required examples. In fact, in view of the flexibility of the constructions in \cite{MuSv} we 
could also arrange it so that the Euler-Lagrange system $-\mathrm{div}F^{\prime}_{j}( \nabla v) =0$ admits, say compactly 
supported Lipschitz maps that are nowhere $\CC^1$ as weak solutions. It is thus clear that all kinds of behaviour of 
quasiconvex integrands of $p$-growth that play out in bounded sets of matrix space can be reproduced by integrands 
satisfying the hypotheses $\mathrm{(H0)}$, $\mathrm{(H1)}$, $\mathrm{(H2)}$. In particular, in view of the nonconvex 
nature of the variational problems it becomes relevant to investigate the regularity of various classes of local 
minimizers as done in the $p$-growth case in \cite{KT,CaNa,Sz}. We leave this for future investigations and focus in the 
present paper entirely on absolute minimizers in the sense of (\ref{intro5}).

\subsection{Extremality of minimizers}

The following result is closely related to \cite[Theorem 3.7]{Anz}, but it concerns more general 
integrands that are not covered there.

\begin{lemma}\label{extremal}
Assume that $F \colon \M \to \R$ is $\CC^1$, rank-one convex and that $|F(z)| \leq L(|z|+1)$ holds for all $z \in \M$.
Then for any local minimizer $u \in \BV ( \Omega , \RN )$ of the variational integral 
$\mathfrak{F}(v, \Omega ) = \int_{\Omega} \! F(Dv)$ we have that $F^{\prime}(\nabla u) \in \LL^{\infty}( \Omega , \M )$
and
\begin{equation}\label{extrem}
-\int_{\Omega} \! F^{\infty}(D^{s}\varphi ) \leq \int_{\Omega} \! F^{\prime}(\nabla u)[ \nabla \varphi ] \, \dd x \leq 
\int_{\Omega} \! F^{\infty}(-D^{s}\varphi )
\end{equation}
holds for all $\varphi \in \BV_{0}( \Omega , \RN )$. In particular, $F^{\prime}( \nabla u)$ is row-wise divergence free.
\end{lemma}

\begin{proof}
First we recall that linear growth and rank-one convexity combine to give Lipschitz continuity
(\ref{lip}), hence the matrix valued map $F^{\prime}( \nabla u)$ is bounded. Next, for $\varphi \in \BV ( \Omega ,\RN )$ 
with compact support in $\Omega$ and each $\varepsilon \geq 0$ we put $\mu := |D^{s}(u+\varepsilon \varphi )|+|D^{s}u|+|D^{s}\varphi |$.
Then we may write
$$
F^{\infty} \bigl( D^{s}(u+\varepsilon \varphi ) \bigr) - F^{\infty} \bigl( D^{s}u \bigr) = \biggl( F^{\infty} \left( 
\frac{\dd D^{s}u}{\dd \mu} + \varepsilon \frac{\dd D^{s}\varphi}{\dd \mu} \right) - F^{\infty} \left( \frac{\dd D^{s}u}{\dd \mu} \right) 
\biggr) \mu .
$$
Here we have according to \cite{Alberti} that
$$
\mathrm{rank} \biggl( \frac{\dd D^{s}\varphi}{\dd \mu} \biggr) \leq 1 \quad \mu \mbox{--a.e.}
$$
and thus from \cite[Lemma 2.5]{KK} and the assumptions on $F$ we infer that
$$
F^{\infty}\bigl( D^{s}(u+\varepsilon \varphi) \bigr) - F^{\infty} \bigl( D^{s}u \bigr) \leq \varepsilon F^{\infty} \bigl( D^{s}\varphi \bigr).
$$
Consequently, by local minimality:
\begin{eqnarray*}
0 &\leq& \int_{\Omega} \! F \bigl( (D(u+\varepsilon\varphi) \bigr) - \int_{\Omega} \! F \bigl( Du \bigr)\\
&\leq& \int_{\Omega} \! \int_{0}^{1} \! F^{\prime} (\nabla u +t\varepsilon \nabla \varphi )[\varepsilon \nabla \varphi ] \, \dd t \, \dd x 
+\varepsilon \int_{\Omega} \! F^{\infty} \bigl( D^{s}\varphi \bigr) ,
\end{eqnarray*}
and hence, invoking the Lipschitz bound and Lebesgue's dominated convergence theorem, we arrive at 
$$
0 \leq \int_{\Omega} \! F^{\prime}(\nabla u)[ \nabla \varphi ] \, \dd x +\int_{\Omega} \! F^{\infty} \bigl( D^{s}\varphi \bigr) .
$$
Finally, we extend the above inequality by continuity to hold for all $\varphi \in \BV_{0}( \Omega , \RN )$.
\end{proof}

\section{Boundedness of minimizing sequences and strong quasiconvexity}\label{minseq}

\begin{proposition}\label{sharpening}
Assume $F\colon \M \to \R$ is a continuous integrand of linear growth, let $\Omega \subset \Rn$ be a bounded Lipschitz domain and
$g \in \WW^{1,1}( \Omega , \RN )$. Then minimizing sequences for the variational problem
\begin{equation}\label{vp1}
\inf_{u \in \WW^{1,1}_{g}( \Omega , \RN )} \int_{\Omega} \! F(\nabla u ) \, \dd x 
\end{equation}
are all bounded in $\WW^{1,1}$ if and only if there exist $\ell > 0$ and $z_{0} \in \M$ such that $F-\ell E$ is quasiconvex at $z_0$.
\end{proposition}

\begin{proof}
The \emph{if} part follows from \cite[Theorem 1.1]{CK} and to prove the \emph{only if} part we must adapt the proofs from \cite{CK}. 
We proceed in three steps.

\noindent
\emph{Step 1.} Let $X$ be a simplex in $\Rn$ and $z_{0} \in \M$. We show that if all minimizing sequences for the variational problem 
(\ref{vp1}) in the special case $\Omega =X$ and $g(x) = z_{0}x$ are $\WW^{1,1}$ bounded, then we can find constants $\alpha > 0$, $\beta \in \R$ 
depending only on $F$, $X$, $z_0$, so
\begin{equation}\label{meancoercive}
\int_{X} \! F( z_{0}+\nabla \varphi ) \, \dd x \geq \int_{X} \! \biggl( \alpha | \nabla \varphi | + \beta \biggr) \, \dd x
\end{equation}
holds for all $\varphi \in \WW^{1,1}_{0}(X , \RN )$. We express this by saying that $F$ is mean coercive, and recall from 
\cite[Theorem 1.1]{CK} that this is a property of $F$ (so that we have a bound like (\ref{meancoercive}) for any bounded Lipschitz
domain $\Omega$ and any $g \in \WW^{1,1}( \Omega , \RN )$ with $\alpha$, $\beta$ now depending on $F$, $\Omega$, $g$).
Following \cite{CK} we consider the auxiliary function
$$
\Theta (t) = \inf \left\{ \dashint_{X} \! F(z_{0}+\nabla \varphi ) \, \dd x : \, \varphi \in \WW^{1,1}_{0}(X, \RN ), \, 
\dashint_{X} \! | \nabla \varphi | \, \dd x \geq t \right\} \quad (t \geq 0)
$$
Because $F$ has linear growth the $\WW^{1,1}$ boundedness of minimizing sequences clearly implies that $\Theta$ is a real-valued 
non-decreasing function. According to \cite[Proposition 3.2]{CK} it is also convex. Consequently, if $\Theta$ is bounded from above, then it must
be constant: $\Theta (t) \equiv \theta$ for all $t \geq 0$, where $\theta \in \R$. But this is impossible as it leads to the
existence of minimizing sequences for (\ref{vp1}) that are not bounded in $\WW^{1,1}$. Hence $\Theta$ is not bounded from above, and so by convexity
we conclude that for some constants $\alpha > 0$, $\beta \in \R$ we must have $\Theta (t) \geq \alpha t + \beta$ for all $t \geq 0$.
Unravelling the definitions we have shown that (\ref{meancoercive}) holds.

\noindent
\emph{Step 2.} We show that if all minimizing sequences for (\ref{vp1}) are $\WW^{1,1}$ bounded, then $F$ is mean coercive.
For this it is easiest to argue by contradiction: Assume that all minimizing sequences for (\ref{vp1}) are $\WW^{1,1}$ bounded, but
that $F$ is not mean coercive. The former, taken together with the linear growth of $F$, means in particular that
$$
m := \inf_{u \in \WW^{1,1}_{g}( \Omega , \RN )} \int_{\Omega} \! F(\nabla u ) \, \dd x \in \R .
$$
The latter allows us by Step 1 to conclude that for any simplex $X \subset \Rn$ and any 
$z_{0} \in \M$, the variational problem (\ref{vp1}) with $\Omega =X$ and $g(x) = z_{0}x$ admits a minimizing sequence that is 
unbounded in $\WW^{1,1}$. Fix a polygonal open subset $\Omega^{\prime} \Subset \Omega$ and note that since $F$ is continuous and of 
linear growth, the functional $v \mapsto \int_{\Omega} \! F(\nabla v) \, \dd x$ is continuous on $\WW^{1,1}( \Omega , \RN )$.
By density of piecewise affine maps in $\WW^{1,1}$ we can therefore find a minimizing sequence $( u_j )$ for (\ref{vp1}) such
that each restriction $u_{j}|_{\Omega^{\prime}}$ is piecewise affine. Let $\tau_j$ be the regular and finite triangulation of $\Omega^{\prime}$
so that $u_{j}$ is affine on each simplex of $\tau_j$. We apply the existence of $\WW^{1,1}$ unbounded minimizing sequences for (\ref{vp1})
for each $\Omega = X \in \tau_{j}$, $z_{0} = \nabla u_{j}|_{X}$ to find $\varphi_{j,X} \in \WW^{1,1}_{0}(X, \RN )$ so
$$
j\Leb (X) < \int_{X} \! | \nabla \varphi_{j,X}| \, \dd x \mbox{ and } \int_{X} \! F(\nabla u_{j}+\nabla \varphi_{j,X}) \, \dd x \leq 
\int_{X} \bigg( F(\nabla u_{j}) + \frac{1}{j} \biggr) \, \dd x.
$$
Defining $v_{j} := u_{j}+\sum_{X \in \tau_{j}} \varphi_{j,X}$, where we extend each $\varphi_{j,X}$ by $0 \in \RN$ off $X$, we have a
$\WW^{1,1}$ unbounded minimizing sequence for (\ref{vp1}), a contradiction that finishes the proof of Step 2.

\noindent
\emph{Step 3.} Conclusion from (\ref{meancoercive}). We may assume that $\Leb (X)=1$.
Now $E(z) \leq |z|$ for $z \in \M$, so if we take $\ell \in (0,\alpha )$, put $c=\alpha -\ell$
and $G=F-\ell E$, then (\ref{meancoercive}) yields
\begin{equation}\label{coer1}
\int_{X} \! G(\nabla \varphi ) \, \dd x \geq c\int_{X} \! | \nabla \varphi | \, \dd x + \beta
\end{equation}
for all $\varphi \in \WW^{1,1}_{0}(X,\RN )$. Recalling the Dacorogna formula for the quasiconvex envelope (see \cite{Dacorogna}
and the discussion in \cite{CK}) we take a sequence $( \varphi_j )$ in $\WW^{1,1}_{0}(X, \RN )$ so
$$
\int_{X} \! G(\nabla \varphi_j ) \, \dd x \to G^{\mathrm{qc}}(0),
$$
the quasiconvex envelope of $G$ at $0$. Obviously, $G^{\mathrm{qc}}(0) \geq \beta$, so $G^{\mathrm{qc}}$ is a real-valued quasiconvex integrand
satisfying $G^{\mathrm{qc}} \leq G$. Because $G$ has linear growth, so does $G^{\mathrm{qc}}$ (see \cite{CK}). The probability measures
$\nu_j$ on $\M$ defined for Borel sets $A \subset \M$ by
$$
\nu_j (A) := \Leb \biggl( X \cap (\nabla \varphi_{j})^{-1}(A) \biggr)
$$
all have centre of mass at $0$ and uniformly bounded first moments: 
$$
c\int_{\M} \! |z| \, \dd \nu_k + \beta \leq \sup_{j} \int_{X} \! G( \nabla \varphi_{j}) \, \dd x < \infty 
$$
for $k \in \N$. But then Banach-Alaoglu's theorem applied in $\CC_{0}( \M )^{\ast}$ yields a subsequence (not relabelled) and 
$\nu \in \CC_{0}( \M )^{\ast}$ such that $\nu_{j} \wstar \nu$ in $\CC_{0}( \M )^{\ast}$. It is not hard to see that $\nu$ must again
be a probability measure on $\M$, and
$$
\int_{\M} \! |z| \, \dd \nu \leq \liminf_{j \to \infty} \int_{\M} \! |z| \, \dd \nu_j < \infty .
$$
Since $G-G^{\mathrm{qc}} \geq 0$ is continuous we get by routine means that
$$
0\leq \int_{\M} \! \biggl( G-G^{\mathrm{qc}} \biggr) \, \dd \nu \leq \liminf_{j \to \infty} 
\int_{\M} \! \biggl( G-G^{\mathrm{qc}} \biggr) \, \dd \nu_{j}=0,
$$
and thus $G=G^{\mathrm{qc}}$ on the support of $\nu$. This completes the proof.
\end{proof}

\begin{remark}\label{cvgenergy}
It is not difficult to show that under assumptions $\mathrm{(H1)}$, $\mathrm{(H2)}$ we have a \emph{principle of convergence of energies}
in the sense that if $u_{j} \wstar u$ in $\BV$, $u_{j}|_{\partial \Omega} \to u|_{\partial \Omega}$ in $\LL^{1}(\partial \Omega , \RN )$ and
$$
\int_{\Omega} \! F(Du_{j}) \to \int_{\Omega} \! F(Du),
$$
then $u_j \to u$ in the area-strict sense in $\BV$. We do not give the details here and intend to return to this in a more general framework
elsewhere.
\end{remark}

\section{Proof of Theorem \ref{thm:main}}\label{sec:main}
We split the proof into five steps, each of which is presented in a subsection.

\subsection{Bounds for shifted integrands} 
For a $\CC^2$ integrand $F \colon \M \to \R$ we define for each $w \in \M$ the shifted integrand $F_{w} \colon \M \to \R$ by
\begin{eqnarray}\label{shifted}
F_{w}(z) &=& F(z+w)-F(w)-F^{\prime}(w)[z]\nonumber\\
&=& \int_{0}^{1} \! (1-t)F^{\prime \prime}(w+tz)[z,z] \, \dd t
\end{eqnarray}
We use the same notation for shifted versions $E_w$ of the reference integrand $E$, and record the
following elementary result for later reference.
\begin{lemma}\label{elemE}
For $w$, $z \in \M$ we have (with obvious interpretation for $w = 0$ or $z=0$)
\begin{equation}\label{elemE1}
E^{\prime \prime}(w)[z,z] = \frac{1+|w|^{2}-|w|^{2}\left( \tfrac{w}{|w|} \cdot \tfrac{z}{|z|} \right)^{2}}{\bigl( 1+|w|^{2} \bigr)^{\frac{3}{2}}}|z|^{2}
\end{equation}
and
\begin{equation}\label{elemE2}
E_{w}(z) \geq 2^{-4}\bigl( 1+|w|^{2} \bigr)^{-\frac{3}{2}}E(z).
\end{equation}
\end{lemma}
The proof of this result is straightforward and is omitted here.
Next, we record the following elementary properties that $F_w$ inherits from $F$:
\begin{lemma}\label{propshift}
Suppose $F \colon \M \to \R$ satisfies $\mathrm{(H0)}$, $\mathrm{(H1)}$, $\mathrm{(H2)}$.
For each $m>0$ there exists a constant $c=c(m) \in [1,\infty )$ with the following properties.
Fix $w \in \M$ with $|w| \leq m$. Then
\begin{equation}\label{sbound1}
|F_{w}(z)| \leq cLE(z) , \quad \quad |F^{\prime}_{w}(z)| \leq cL \min \{ |z| ,1 \} , 
\end{equation}
\begin{equation}\label{sbound2}
|F^{\prime \prime}_{w}(0)z-F^{\prime}_{w}(z)| \leq cLE(z)
\end{equation}
holds for all $z \in \M$, 
\begin{equation}\label{sqc}
\int_{B} \! F_{w}(\nabla \varphi (x)) \, \dd x \geq \tfrac{\ell}{c}\int_{B} \! E( \nabla \varphi (x)) \, \dd x
\end{equation}
holds for all $\varphi \in \WW^{1,1}_{0}(B, \RN )$ and
\begin{equation}\label{src}
F^{\prime \prime}(w)[y \otimes x, y \otimes x] \geq \tfrac{\ell}{c}|y|^{2}|x|^{2}
\end{equation}
holds for all $x \in \Rn$, $y \in \RN$.
\end{lemma}

\begin{proof}
For the bounds (\ref{sbound1}) and (\ref{sbound2}) we distinguish the cases $|z| \leq 1$ and $|z| > 1$.
The bounds in (\ref{sbound1}) follow then easily from the definition of $F_w$ and (\ref{lip}).
We leave the details of this to the reader, and instead focus on (\ref{sbound2}). Here we have for $|z| \leq 1$ that 
\begin{eqnarray*}
|F^{\prime \prime}_{w}(0)z-F^{\prime}_{w}(z)| &\leq& \int_{0}^{1} \! \bigl| 
F^{\prime \prime}(w)-F^{\prime \prime}(w+tz) \bigr| \, \dd t |z|\\
&\leq& \lip (F^{\prime \prime},B_{m+1}(0))|z|^{2}
\end{eqnarray*}
and the latter is finite for each fixed $m$ by hypothesis (H0). Next, for $|z| > 1$ we use (\ref{lip}) to estimate:
\begin{eqnarray*}
|F^{\prime \prime}_{w}(0)z-F^{\prime}_{w}(z)| &\leq& |F^{\prime \prime}(w)||z|+cL\\
&\leq& \bigl( \sup_{|v| \leq m}|F^{\prime \prime}(v)| +cL \bigr)|z| ,
\end{eqnarray*}
and so we deduce (\ref{sbound2}) from (\ref{minb}). Finally we turn to the quasiconvexity condition (\ref{sqc}). From
(H2) we get
$$
\int_{B} \! F_{w}(\nabla \varphi ) \, \dd x \geq \ell \int_{B} \! E_{w}(\nabla \varphi ) \, \dd x
$$
and so from (\ref{elemE2}) we get (\ref{sqc}) with $c=2^{4}(1+m^{2})^{\frac{3}{2}}$. Finally, since quasiconvexity implies 
rank one convexity, (H2) yields in particular that
\begin{eqnarray*}
F^{\prime \prime}(w)[y \otimes x, y \otimes x] &\geq& \ell E^{\prime \prime}(w)[y \otimes x, y \otimes x]\\
&\stackrel{(\ref{elemE1})}{\geq}& \frac{\ell}{(1+m^{2})^{\frac{3}{2}}}|y|^{2}|x|^{2} ,
\end{eqnarray*}
which of course implies (\ref{src}). 
\end{proof}

\subsection{Caccioppoli inequality of the second kind}\label{caccioppoli2}
This is an important part of the proof, and in fact it is the only place where both the quasiconvexity and minimality 
assumptions are used. However, the proof in the considered linear growth case does not differ much from the usual ones
as it follows that given by Evans in \cite{Evans} and relies crucially on Widman's hole filling trick \cite{Widman}.
\begin{proposition}\label{caccioppoli}
Suppose $F \colon \M \to \R$ satisfies $\mathrm{(H0)}$, $\mathrm{(H1)}$, $\mathrm{(H2)}$ and that $u \in \BV (\Omega , \RN )$
is a minimizer. Then each $m>0$ there exists a constant $c=c(m,n,N,\tfrac{L}{\ell}) \in [1,\infty )$ with the following property.
Let $a \colon \Rn \to \RN$ be an affine mapping with $| \nabla a| \leq m$ and $B_{R}(x_{0}) \subset \Omega$.
Then we have
\begin{equation}\label{caccio}
\int_{B_{\frac{R}{2}}(x_{0})} \! E(D(u-a)) \leq c\int_{B_{R}(x_{0})} \! E \left( \frac{u-a}{R} \right) \, \dd x.
\end{equation}
\end{proposition}

\begin{proof}
Denote $\tF = F_{\nabla a}$ and $\tu = u-a$. Observe that $\tu$ is minimizing the integral functional
corresponding to the shifted integrand $\tF$. Fix two radii $\tfrac{R}{2} < r < s < R$, and let $\rho \colon \Omega \to \R$
be a Lipschitz cut-off function satisfying $\mathbf{1}_{B_{r}} \leq \rho \leq \mathbf{1}_{B_{s}}$ and 
$| \nabla \rho | \leq \frac{1}{s-r}$. Put $\varphi = \rho \tu$ and $\psi = (1-\rho )\tu$. For a standard smooth mollifier
$( \phi_{\varepsilon})$ we let $\varphi_{\varepsilon} = \rho (\phi_{\varepsilon} \ast \tu )$, so that 
$\varphi_{\varepsilon} \in \WW^{1,1}_{0}(B_{s},\RN )$. Hence by the consequence (\ref{sqc}) of the quasiconvexity assumption (H2) we get
$$
\frac{\ell}{c}\int_{B_{s}} \! E(\nabla \varphi_{\varepsilon}) \, \dd x \leq 
\int_{B_{s}} \! \tF (\nabla \varphi_{\varepsilon}) \, \dd x .
$$
Observe that as $\varepsilon \searrow 0$, $\varphi_{\varepsilon} \to \varphi$ in $\LL^{1}$ and (since $\rho = 0$ on $\partial B_s$) that
$$
\int_{B_{s}} \! E(\nabla \varphi_{\varepsilon}) \, \dd x \to  \int_{B_{s}} \! E(D\varphi ).
$$
We can therefore employ Lemma \ref{Econt} to find, by taking $\varepsilon \searrow 0$ in the above inequality, 
$$
\frac{\ell}{c}\int_{B_{s}} \! E(D\varphi ) \leq \int_{B_{s}} \! \tF (D\varphi ).
$$
Consequently, we have using minimality of $\tu$, (\ref{sbound1}), convexity of $E$ and (\ref{minb}):
\begin{eqnarray*}
\frac{\ell}{c}\int_{B_{r}} \! E(D\tu ) &\leq& \int_{B_{s}} \! \tF (D\tu ) + \int_{B_{s}} \! \tF (D\varphi ) - \int_{B_{s}} \! \tF (D\tu )\\
&\leq& \int_{B_{s}} \! \tF (D\psi ) + \int_{B_{s}} \! \tF (D\varphi ) - \int_{B_{s}} \! \tF (D\tu )\\
&\leq& cL\int_{B_{s}\setminus B_{r}} \! E(D\tu ) + cL\int_{B_{s}\setminus B_{r}} \! E(\rho D\tu + \tu \otimes \nabla \rho )\\  
&& +cL\int_{B_{s}\setminus B_{r}} \! E((1-\rho )D\tu - \tu \otimes \nabla \rho )\\
&\leq& 5cL\int_{B_{s}\setminus B_{r}} \! E(D\tu ) +4cL\int_{B_s} \! E \left( \frac{\tu}{s-r} \right) \, \dd x.
\end{eqnarray*}
We fill the hole whereby on denoting $\theta = 5cL/(5cL + \tfrac{\ell}{c}) \in (0,1)$ we arrive at
\begin{eqnarray*}
\int_{B_r} \! E(D \tu ) &\leq& \theta \int_{B_s} \! E(D \tu ) + \theta 
\int_{B_s} \! E\left( \frac{\tu}{s-r} \right) \, \dd x\\
&\leq&  \theta \int_{B_s} \! E(D \tu ) + \theta 
\int_{B_R} \! E\left( \frac{\tu}{s-r} \right) \, \dd x.
\end{eqnarray*}
The conclusion now follows in a standard way from the iteration Lemma \ref{iterate} below.
\end{proof}

\begin{lemma}\label{iterate}
Let $\theta \in (0,1)$, $A \geq 1$ and $R>0$. Assume that $\Phi$, $\Psi \colon (0,R] \to \R$ are
nonnegative functions, that $\Phi$ is bounded, $\Psi$ is decreasing with $\Psi (\tfrac{t}{2}) \leq A\Psi (t)$ for
all $t \in (0,R]$ and that
\begin{equation}\label{ineqiterate}
\Phi (r) \leq \theta \Phi (s) + \Psi (s-r)
\end{equation}
holds for all $r$, $s \in [\tfrac{R}{2},R]$ with $r<s$. Then there exists a constant $C=C(\theta , A) >0$
such that
\begin{equation}\label{iterated}
\Phi \left( \tfrac{R}{2} \right) \leq C\Psi (R).
\end{equation}
\end{lemma}
The proof follows closely that of, for instance, \cite[Lemma 6.1]{Giusti} and so we leave the details to the reader.
\bigskip

\noindent
As mentioned in the Introduction it is not possible to use the Poincar\'{e}-Sobolev inequality to get a reverse H\"{o}lder
inequality from which higher integrability can be deduce by use of Gehring's Lemma. However, the Caccioppoli inequality
(\ref{caccio}) still encodes some compactness as can be seen from the following remark that is stated in terms of
Young measures and where we use the terminology from \cite{KR2}. The reader can get a good overview of this general
formalism and other developments in the calculus of variations context in the recent monograph \cite{FR}.

\begin{remark}\label{caccpt}
Let $(u_j )$ be a sequence in $\BV ( \Omega , \RN )$ satisfying the Caccioppoli inequality (\ref{caccio}) above uniformly:
for each $m > 0$ there exists a constant $c_m$ (independent of $j$) such that for any affine map $a \colon \Rn \to \RN$
with $| \nabla a| \leq m$ and any ball $B_{R}=B_{R}(x_0 ) \subset \Omega$ we have
\begin{equation}\label{caccj}
\int_{B_{\frac{R}{2}}(x_{0})} \! E(D(u_{j}-a)) \leq c_m \int_{B_{R}(x_{0})} \! E \left( \frac{u_{j}-a}{R} \right) \, \dd x.
\end{equation}
If $(u_j )$ is bounded in $\BV ( \Omega , \RN )$, then (any subsequence) admits a subsequence (not relabelled) that
converges weakly$\mbox{}^\ast$ in $\BV$ to a map $u \in \BV ( \Omega , \RN )$ and whose derivatives $Du_j$ generate
a Young measure $\nu = \bigl( ( \nu_x )_{x \in \Omega} , \lambda , ( \nu^{\infty}_{x} )_{x \in \overline{\Omega}} \bigr)$.
The compactness encoded in (\ref{caccj}) amounts to
$$
\nu_{x} = \delta_{\nabla u(x)} \quad \Leb\mbox{-a.e.} \quad \mbox{ and } \quad |D^{s}u| \leq \lambda \lfloor \Omega \leq c|D^{s}u| ,
$$
where $c=c(n)c_0$.
\end{remark}

\begin{proof}
The existence of the subsequence with the asserted properties follows from \cite[Theorem 8]{KR2}. Thus we have
for some subsequence (not relabelled), $u \in \BV ( \Omega , \RN )$ and Young measure $\nu$ that
$$
u_j \wstar u \mbox{ in } \BV \, \mbox{ and } \, Du_j \toY \nu.
$$
By a result of Calder\'{o}n and Zygmund \cite[Theorem 3.83]{AFP}, $u$ is approximately differentiable $\Leb$ almost everywhere.
Let $x_0 \in \Omega$ be such a point and take $a(x) = u(x_{0})+\nabla u(x_{0})(x-x_{0})$. With $m=| \nabla u(x_{0})|$
we get from (\ref{caccj}) on a ball $B_{2r}=B_{2r}(x_{0}) \subset \Omega$ after taking $j \nearrow \infty$:
$$
\int_{B_{r}} \! \int_{\M} \! E \bigl( \cdot - \nabla u(x_{0}) \bigr) \, \dd \nu_{x} \, \dd x + \lambda (B_{r}) \leq 
c_{m} \int_{B_{2r}} \! E \left( \frac{u-a}{r} \right) \, \dd x.
$$
Divide by $\Leb (B_r )$ and take $r \searrow 0$ to get by Lebesgue's differentiation theorem
$$
\int_{\M} \! E  \bigl( \cdot - \nabla u(x_{0}) \bigr) \, \dd \nu_{x_0} + \frac{\dd \lambda}{\dd \Leb}(x_{0}) \leq 0
$$
for $\Leb$ almost all such $x_0$. But then both terms on the left-hand side must be $0$, and so, using the strict convexity 
of $E$ for the first term, we conclude that
$$
\nu_{x_{0}} = \delta_{\nabla u(x_{0})} \, \mbox{ for $\Leb$-a.e. } x_0  \, \mbox{ and } \, \lambda \perp \Leb .
$$
We always have $|D^{s}u| \leq \lambda \lfloor \Omega$ (see for instance \cite{KR2}). For the upper bound we fix an arbitrary ball 
$B_{2r} \subset \Omega$ take $a=u_{B_{2r}}$ above and pass to the limit whereby
$$
\int_{B_{r}} \! E \bigl( \nabla u \bigr) \, \dd x + \lambda ( B_{r}) \leq c_{0}\int_{B_{2r}} \! E \left( \frac{u-u_{B_{2r}}}{r} \right) \, \dd x
$$
results. Using that $E(z) \leq |z|$ for all $z$ and Poincar\'{e}'s inequality on the right-hand side we get
$$
\int_{B_{r}} \! E \bigl( \nabla u \bigr) \, \dd x + \lambda ( B_{r}) \leq c_{0}c |Du|(B_{2r}).
$$
Put $\mu = c_{0}c|Du|$; then $\lambda (B) \leq \mu (2B)$ for any ball $B$ for which
$2B \subset \Omega$. We reformulate this bound in terms of cubes as follows. For a closed ball $\overline{B}$ we let $Q$ denote the largest 
closed cube with sides parallel to the coordinate axes that is contained in $\overline{B}$ and we let $\hat{Q}$ denote the smallest such cube that
contains $s\overline{B}$ for some fixed $s>2$. Then $Q$ and $\hat{Q}$ are concentric and the sidelengths satisfy $\ell (\hat{Q}) = s\sqrt{n}\ell (Q)$. 
Clearly given a cube $Q$ with sides parallel to the coordinate axes, the cube $\hat{Q}$ just described is uniquely determined and we must in particular
have $\lambda (Q) \leq \mu ( \hat{Q})$ for all cubes $Q$ with $\hat{Q} \subset \Omega$. Now fix a closed cube $Q \subset \Omega$ and
consider the system of its $Q$-dyadic subcubes at level $k \in \N$: 
$$
Q = \bigcup_{j=1}^{2^{k}} Q^{(k)}_{j}.
$$
For $k$ sufficiently large we have that each $\hat{Q}^{(k)}_{j} \subset \Omega$ since $\ell (\hat{Q}^{(k)}_{j}) = s\sqrt{n}2^{-k}\ell (Q)$, 
$Q^{(k)}_{j} \subset \hat{Q}^{(k)}_{j} \cap Q \Subset \Omega$, and so $\lambda (Q^{(k)}_{j}) \leq \mu (\hat{Q}^{(k)}_{j})$, hence
$$
\lambda (Q) \leq \sum_{j=1}^{2^k} \mu (\hat{Q}^{(k)}_{j}) = \int \! \sum_{j=1}^{2^k} \mathbf{1}_{\hat{Q}^{(k)}_{j}} \, \dd \mu \leq c(n,s) \mu (Q_k ),
$$
where $Q_k = \bigcup_j \hat{Q}^{(k)}_{j}$ and we used that the family of cubes satisfies a uniform bounded overlap property. Taking $k \nearrow \infty$
we arrive at $\lambda (Q) \leq c(n,s) \mu (Q )$. Now since the cube $Q \subset \Omega$ was arbitrary and $\lambda$ is singular the proof is complete.
\end{proof}

\subsection{Approximation by harmonic maps}\label{approxharm}

We turn to the announced approximation by harmonic maps. This step, where the minimizer is compared with 
the solution to a suitably linearized problem, is standard fare in partial regularity proofs and goes back to the works 
\cite{Almgren1,Almgren2,DeGiorgi1}. However, due to the $\LL^1$ set-up the usual ways of implementing this linearization 
(such as for instance \cite{AcerbiFusco1,AcerbiFusco2,CarozzaFuscoMingione,DGK,
DLSV,Giusti,Hamburger}) do seem to require modification. Fortunately, our variant is quite straightforward and proceeds 
by explicit construction of a test map that yields the required estimate. In fact, we believe that when this construction is
applied in the cases covered previously in the literature, it also offers a useful alternative argument there.

Because the approximation result is achieved by a linearization argument it is more natural if we also replace the key assumptions 
(H2) and minimality by their corresponding linearizations. More precisely, we shall replace the quasiconvexity hypothesis (H2) on the 
integrand $F$ by its linearization, namely the corresponding weaker rank-one convexity hypothesis:
$$
(\mathrm{H2W}) \hspace{1cm} z \mapsto F(z)-\ell E(z) \, \mbox{ is rank-one convex}.
$$
Instead of assuming that $u$ is a minimizer, we assume that $u \in \BV (\Omega , \RN )$ satisfies the extremality condition (\ref{extrem}).
We then have the following:

\begin{proposition}\label{approxharmprop}
Let $F \colon \M \to \R$ satisfy $\mathrm{(H0)}$, $\mathrm{(H1)}$ and $\mathrm{(H2W)}$ and assume that $u \in \BV (\Omega , \RN )$ 
satisfies (\ref{extrem}). Fix a number $m>0$. For any affine map $a \colon \Rn \to \RN$ with $| \nabla a| \leq m$ and each ball
$B=B_{R}(x_{0}) \subset \Omega$ so that $u|_{\partial B} \in \BV (\partial B , \RN )$ and $|Du|( \partial B)=0$ the elliptic system
\begin{equation}\label{defh}
\left\{
\begin{array}{ll}
-\mathrm{div} F^{\prime \prime}(\nabla a)\nabla h = 0 & \mbox{ in } B\\
h=u|_{\partial B} & \mbox{ on } \partial B,
\end{array}
\right.
\end{equation}
admits a unique solution $h \in \WW^{1,1}(B, \RN )$. This solution $h$ satisfies
\begin{equation}\label{boundharm}
\left( \dashint_{B} \! | \nabla h-\nabla a|^{p} \, \dd x \right)^{\frac{1}{p}} \leq c \dashint_{\partial B} \! |D_{\tau}(u-a) |
\end{equation}
for exponents $p \in (1,2)$ when $n=2$ and $p \in (1,\tfrac{n}{n-1}]$ when $n \geq 3$ and a corresponding constant $c=c(n,N,m,p,\tfrac{L}{\ell})$. 
Moreover, for each exponent $q \in (1,\tfrac{n}{n-1})$,
\begin{equation}\label{keyapprox}
\dashint_{B} \! E \left( \frac{u-h}{R} \right) \, \dd x \leq C\left( \dashint_{B} \! E \bigl( D(u-a) \bigr) \right)^{q}, 
\end{equation}
where $C=C(m,n,N,q,L,\ell )$.
\end{proposition}

\begin{proof} 
We give the details for the case $n \geq 3$ only and leave it to the reader to check that the same proof applies for $n=2$, where
the only difference is that Lemma \ref{bvembedding2} is used instead of Lemma \ref{bvembedding}.

Let $x_{0} \in \Omega$ and fix a number $m>0$. By virtue of Lemma \ref{bvrestrict} $\mathscr{L}^1$ a.e. 
radii $R \in (0, \mathrm{dist}(x_{0}, \partial \Omega ))$ have the property that $u |_{\partial B} \in \BV (\partial B , \RN )$ and $|Du|(\partial B)=0$,
where we wrote $B=B_{R}(x_{0})$. 
We fix such a radius $R$ and write as already indicated $B=B_{R}(x_{0})$. For an affine map $a \colon \Rn \to \RN$ with
$| \nabla a| \leq m$ we put as in the previous subsection $\tF = F_{\nabla a}$ and $\tu = u-a$. Clearly, $\tu |_{\partial B} \in \BV (\partial B , \RN )$
remains true. From (H0) and (H2W) we infer that
\begin{equation}\label{LH}
\tF^{\prime \prime}(0)[ y \otimes x ,y \otimes x ] \geq \tfrac{\ell}{c}| y |^{2} | x |^{2} \quad \forall y \in \RN, \, \forall 
x \in \Rn \quad \mbox{ and } \quad |\tF^{\prime \prime}(0)| \leq c,
\end{equation}
where $c=c(m)>0$ is a constant that as indicated depends on $m$. 

As is customary in this context, we make use of (\ref{extrem}) in 
a linearized form by rewriting it for $\varphi \in \CC^{\infty}_{c}(B, \RN )$ as
\begin{eqnarray*}
\int_{B} \! \tF^{\prime \prime}(0)[ D\tu , \nabla \varphi ] 
&=& \int_{B} \! \tF^{\prime \prime}(0)[ D^{s}\tu , \nabla \varphi ]\\
&& +\int_{B} \! \langle \tF^{\prime \prime}(0)\nabla \tu -\tF^{\prime}(\nabla \tu ),\nabla \varphi \rangle \, \dd x\\
&\stackrel{(\ref{sbound2})}{\leq}& c\int_{B} \! |D^{s}\tu | | \nabla \varphi |\\
&& +cL\int_{B} \! E(\nabla \tu )| \nabla \varphi | \, \dd 
x\\
&\leq& c\int_{B} \! E(D\tu )| \nabla \varphi |.
\end{eqnarray*}
It is at this stage we take advantage of the particular choice of radius $R$ whereby $\tu |_{\partial B}$ 
is BV on $\partial B$ and $|D\tu |(\partial B)=0$. The latter ensures that we may extend the above bound by continuity to hold for all 
$\varphi \in (\WW^{1,\infty}_{0} \cap \CC^{1})(B, \RN )$. The former gives in combination with the embedding result of 
Lemma \ref{bvembedding} that $\tu |_{\partial B} \in \WW^{\frac{1}{n},\frac{n}{n-1}}(\partial B, \RN )$
and
$$
\left( \dashint_{\partial B} \! \int_{\partial B} \! \frac{| \tu (x) - \tu (y)|^{\frac{n}{n-1}}}{|x-y|^{n-1+\frac{1}{n-1}}} \, \dd \mathcal{H}^{n-1}(x)
\, \dd \mathcal{H}^{n-1}(y) \right)^{1-\frac{1}{n}} \leq cR^{1-\frac{1}{n}}\dashint_{\partial B} \! |D_{\tau}\tu | 
$$
for a dimensional constant $c=c(n,N)$.
In view of (\ref{LH}) and Theorem \ref{exH} we can then find a unique solution $\tha \in \WW^{1,\frac{n}{n-1}}(B,\RN )$ to the boundary value
problem
\begin{equation}\label{harmcomp}
\left\{
\begin{array}{ll}
-\mathrm{div} \mathbb{A}\nabla \tha = 0 & \mbox{ in } B\\
\tha = \tu & \mbox{ on } \partial B,
\end{array}
\right.
\end{equation}
where $\mathbb{A}= \tF^{\prime \prime}(0)$. In particular we record that
\begin{equation}\label{Esys}
\int_{B} \! \mathbb{A} [ \nabla \tha , \nabla \varphi ] \, \dd x =0
\end{equation} 
holds for all $\varphi \in \WW^{1,n}_{0}(B, \RN )$ and also that the integral estimate (\ref{boundharm}) holds.

Put $\psi = \tu - \tha$ so that $\psi \in \BV_{0}(B,\RN )$ and
\begin{equation}\label{almosth}
\int_{B} \! \mathbb{A}[ \nabla \psi , \nabla \varphi ] \, \dd x \leq c\int_{B} E(D \tu )| \nabla \varphi | 
\end{equation}
holds for all $\varphi \in (\WW^{1,\infty}_{0} \cap \CC^{1})(B, \RN )$, where $c=c(m,L)$. 
We extract information from (\ref{almosth}) by constructing a suitable test map $\varphi$. 
It is convenient to change variables and refer everything to the open unit ball as follows: Put for $x \in \BB := B_{1}(0)$
$$
\Psi (x) = \frac{1}{R}\psi (x_{0}+Rx),\quad \Phi (x) = \frac{1}{R}\varphi (x_{0}+Rx), \quad U(x)=\frac{1}{R}\tu (x_{0}+Rx).
$$
Then (\ref{almosth}) becomes
\begin{equation}\label{almosthB}
\int_{\BB} \! \mathbb{A}[ D\Psi , \nabla \Phi ] \, \dd x \leq c\int_{\BB} \! E(DU )| \nabla \Phi | \quad \forall
\, \Phi \in (\WW^{1,\infty}_{0} \cap \CC^{1})(\BB, \RN ).
\end{equation}
Denote by $T \colon \RN \to \RN$ the truncation mapping defined by
$$
T(y) = \left\{
\begin{array}{ll}
y & \mbox{ if } |y| \leq 1\\
\tfrac{y}{|y|} & \mbox{ if } |y| > 1,
\end{array}
\right.
$$
and consider the elliptic system
\begin{equation}\label{Esystest}
\left\{
\begin{array}{ll}
-\mathrm{div} \, \mathbb{A}\nabla \Phi = T(\Psi ) & \mbox{ in } \BB\\
\Phi = 0 & \mbox{ on } \partial \BB.
\end{array}
\right.
\end{equation}
Evidently the right-hand side is bounded and we have a unique solution $\Phi \in \WW^{1,2}_{0}(\BB , \RN )$. 
From Proposition \ref{exH} it follows that $\Phi$ is 
of Sobolev class $\WW^{2,p}(\BB ,\RN )$ for each exponent $p \in (1, \infty )$ with bound
\begin{equation}\label{CZp}
\int_{\BB} \! | \nabla^{2} \Phi |^{p} \, \dd x \leq C \int_{\BB} \! |T(\Psi )|^{p} \, \dd x
\end{equation}
where $C=C(m,n,N,p,L,\ell )$ is a constant. If we take $p>n$, then we have that $\Phi \in \CC^{1,1-\frac{n}{p}}(\BB, \RN )$ and since 
$(\nabla \Phi )_{\BB}=0$ it follows from Morrey's inequality (see for instance \cite[Sect.~4.5, Th.~3]{EvGa}) that
$$
\| \nabla \Phi \|_{\LL^{\infty}} \leq c\| \nabla^{2} \Phi \|_{\LL^{p}} \leq c \| T(\Psi ) \|_{\LL^p}.
$$ 
In particular, $\Phi \in (\WW^{1,\infty}_{0} \cap \CC^{1})(\BB , \RN )$ so that $\Phi$ indeed qualifies as a test map in (\ref{almosthB}) and then, 
in turn, by approximation, $\Psi \in \BV_{0}(\BB ,\RN )$ qualifies as a test map in (\ref{Esystest}). 
We also note that a simple estimation using (\ref{minb}) yields 
$$
\| T( \Psi ) \|_{\LL^p} \leq c\left( \int_{\BB} \! E(\Psi ) \, \dd x \right)^{\tfrac{1}{p}},
$$
and consequently
\begin{equation}\label{basic}
\| \nabla \Phi \|_{\LL^{\infty}} \leq c \left( \int_{\BB} \! E(\Psi ) \, \dd x \right)^{\tfrac{1}{p}}
\end{equation}
holds for exponents $p \in (n,\infty )$ and corresponding constants $c=c(m,n,N,p,L,\ell )$. We plug this $\Phi$ into (\ref{almosthB});
recalling that $\Psi$ can be used to test (\ref{Esystest}) and that $\mathbb{A}$ is symmetric the following string 
of inequalities results:
\begin{eqnarray*}
\int_{\BB} \! \min \{ | \Psi |^{2},| \Psi | \} \, \dd x &=& \int_{\BB} \! \langle \Psi , T( \Psi ) \rangle \, \dd x\\
&\stackrel{(\ref{Esystest})}{=}& \int_{\BB} \!  \langle \mathbb{A} \nabla \Phi , \nabla \Psi \rangle \, \dd x\\
&=& \int_{\BB} \!  \langle \mathbb{A} \nabla \Psi , \nabla \Phi \rangle \, \dd x\\
&\stackrel{(\ref{almosthB}), (\ref{basic})}{\leq}& c\int_{\BB} \! E(DU)  
\left( \int_{\BB} \! E(\Psi ) \, \dd x \right)^{\frac{1}{p}},
\end{eqnarray*}
and thus (using again (\ref{minb}))
$$
\left( \int_{\BB} \! E(\Psi ) \, \dd x \right)^{1-\frac{1}{p}} \leq c\int_{\BB} \! E(DU).
$$
Hence we have shown that
\begin{equation}\label{pkeyapprox}
\int_{\BB} \! E(\Psi ) \, \dd x \leq C\left( \int_{\BB} \! E(DU) \right)^{q} 
\end{equation}
where $q=p/(p-1) \in (1, \tfrac{n}{n-1} )$ is the dual exponent and $C=C(m,n,N,q,L,\ell )$ is a constant. Finally, we change
back variables $x \mapsto x_{0}+Rx$ and recall that $\psi = \tu - \tha$ whereby (\ref{pkeyapprox}) turns into (\ref{keyapprox})
thus completing the proof. 
\end{proof}

\subsection{Excess decay estimate}\label{excessdecay}
For a map $u \in \BV (\Omega , \RN )$ and a ball $B_{r}(x_{0}) \subset \Omega$ the relevant excess functional is
$$
\mathscr{E}(x_{0},r) = \int_{B_{r}(x_{0})} \! E \bigl( Du-(Du)_{B_{r}(x_{0})} \bigr) .
$$
The goal of this subsection is the following excess decay estimate:
\begin{proposition}\label{edprop}
Suppose $F \colon \M \to \R$ satisfies $\mathrm{(H0)}$, $\mathrm{(H1)}$, $\mathrm{(H2)}$ and that $u \in \BV (\Omega , \RN )$
is a minimizer. Then each $m>0$ and $q \in (1,\tfrac{n}{n-1})$ there exists a constant $c=c(m,q,n,N,\tfrac{L}{\ell})$ with
the following property. For a ball $B_R = B_{R}(x_{0}) \subset \Omega$ such that
\begin{equation}\label{assump1}
|(Du)_{B_{R}}| < m
\end{equation}
and
\begin{equation}\label{assump2}
\dashint_{B_{R}} \! |Du-(Du)_{B_{R}}| \leq 1
\end{equation}
we have that
\begin{equation}\label{keyexcessdecay}
\mathscr{E} (x_{0},\sigma R) \leq c \left( \sigma^{n+2} + \left( \frac{\mathscr{E}(x_{0},R)}{\Leb (B_{R}(x_{0}))} \right)^{q-1} \right) \mathscr{E}(x_{0},R) 
\end{equation}
holds for all $\sigma \in (0,1)$.
\end{proposition}

\begin{proof}
We give the details for the case $n \geq 3$ only and leave it to the reader to check that the same proof applies for $n=2$, where
the only difference is that Lemma \ref{bvembedding2} is used instead of Lemma \ref{bvembedding}. As in the previous subsections we put
$\tu = u-a$ and $\tF = F_{\nabla a}$ and remark that by virtue of our assumptions both results from subsections 3.2 and 3.3 are now available.

In view of Lemma \ref{bvrestrict} we can select $r \in (\frac{9}{10}R,R)$ such that $\tu |_{\partial B_{r}} \in \BV (\partial B_{r}, \RN )$ and
\begin{equation}\label{boundR}
\int_{\partial B_{r}} \! |D_{\tau}(\tu |_{\partial B_{r}})| \leq \frac{20}{R}\int_{B_{R}} \! |D \tu |.
\end{equation}
Now the harmonic map $\tha$ determined at (\ref{harmcomp}) satisfies (\ref{boundharm}) and (\ref{keyapprox}). Let
$A \colon \Rn \to \RN$ be the affine map $A(x)= \tha (x_{0})+\nabla \tha (x_{0})(x-x_{0})$ and put $a_{0}=a+A$.
Then $a_0$ is clearly affine and in order to estimate $| \nabla a_0 |$ we note that according to Lemma \ref{Weyl} we have 
for a constant $c=c(n,N,m,\tfrac{L}{\ell})$:
\begin{eqnarray*}
| \nabla \tha (x_{0})| &\leq& \sup_{B_{\frac{r}{2}}} | \nabla \tha | \leq c\dashint_{B_r} \! | \nabla \tha | \, \dd x\\
&\leq& c\left( \dashint_{B_r} \! | \nabla \tha |^{\frac{n}{n-1}} \, \dd x\right)^{\frac{n-1}{n}}\\
&\stackrel{(\ref{boundharm})}{\leq}& c\dashint_{\partial B_r} \! |D_{\tau}(\tu |_{\partial B_r})|\\
&\stackrel{(\ref{boundR})}{\leq}& \frac{c}{R r^{n-1}}\int_{B_{R}} \! |D \tu |\\
&\leq& c\dashint_{B_{R}} \! |D \tu |.
\end{eqnarray*}
In view of (\ref{assump2}) we therefore have that
\begin{eqnarray*}
| \nabla a_0 | &\leq& |(Du)_{B_{R}}| + c\dashint_{B_{R}} \! | Du-(Du)_{B_{R}}|\\
&<& m+c(m) =: C_{m}
\end{eqnarray*}
holds. For $\sigma \in (0,\frac{1}{5})$ we have by (\ref{qminE})
$$
\int_{B_{\sigma R}} \! E(Du-(Du)_{B_{\sigma R}}) \leq 12 \int_{B_{\sigma R}} \! E(D(u-a_0 )).
$$
Next, we apply the Caccioppoli inequality (\ref{caccioppoli2}) on the ball $B_{2\sigma R}=B_{2\sigma R}(x_{0})$ and with the 
affine map $a_0$ defined above:
$$
\int_{B_{\sigma R}} \! E(D(u-a_0 )) \leq c\int_{B_{2\sigma r}} \! E\left( \frac{u-a_{0}}{2\sigma r}\right) \, \dd x
$$
where $c=c(m)$ is a constant obtained from Proposition \ref{caccioppoli} and estimation of the right-hand side using 
$R \in (\tfrac{9}{10}r,r)$ and (\ref{minb}). 
We combine these bounds and use (\ref{minb}) again twice:
\begin{eqnarray*}
\int_{B_{\sigma R}} \! E(Du-(Du)_{x_{0},\sigma R}) &\leq& C\int_{B_{2\sigma R}} \! 
\left( E\left( \frac{\tu-\tha}{\sigma R}\right) +E\left( \frac{\tha-A}{2\sigma R}\right) \right) \, \dd x\\
&\leq& \frac{c}{\sigma^{2}}\int_{B_{r}} \! E\left( \frac{\tu -\tha}{r}\right) \, \dd x+c\int_{B_{2\sigma R}} \! 
E\left( \frac{\tha -A}{2\sigma R}\right) \, \dd x.
\end{eqnarray*}
Here we have for the first term according to (\ref{keyapprox}) for each exponent $q \in (1,\tfrac{n}{n-1})$ and $C=C(m,n,N,q,L,\ell )$ that
$$
\int_{B_{r}} \! E \left( \frac{\tu -\tha}{r} \right) \, \dd x \leq C\left( \dashint_{B_{r}} \! E(D \tu ) \right)^{q}\Leb (B_{R}). 
$$
The second term is estimated using Lemma \ref{Weyl}. Accordingly we have for $x \in B_{2\sigma R} \subset B_{\frac{r}{2}}$ and in view
of our choice of the affine map $A$:
\begin{eqnarray*}
\frac{|\tha (x)-A(x)|}{\sigma R} &\leq& c\sup_{x \in B_{\frac{r}{2}}} \left( | \nabla^{2}\tha (x) | \frac{|x-x_{0}|^2}{\sigma R} \right)\\
&\stackrel{\text{Lemma }\ref{Weyl}}{\leq}& c\dashint_{B_r} \! | \nabla \tha | \, \dd x \sigma\\
&\stackrel{(\ref{boundharm}), (\ref{boundR})}{\leq}& c\dashint_{B_{R}} \! |Du-(Du)_{B_{R}}| \sigma\\
&\stackrel{(\ref{assump2}), (\ref{smallE})}{\leq}& c\sigma \left( \dashint_{B_{R}} \! E(Du-(Du)_{B_{R}}) \right)^{\frac{1}{2}} .
\end{eqnarray*}
Consequently we have
\begin{eqnarray*}
\int_{B_{2\sigma R}} \! E\left( \frac{\tha -A}{2\sigma R}\right) \, \dd x &\leq& c(\sigma R)^{n}
E \left( \sigma \left( \dashint_{B_{R}} \! E(Du-(Du)_{B_{R}}) \right)^{\frac{1}{2}} \right)\\
&\leq& c\sigma^{n+2}\int_{B_{R}} \! E(Du-(Du)_{B_{R}}),
\end{eqnarray*}
and hence we arrive upon collection of the bounds at (\ref{keyexcessdecay}). Increasing the constant $c$ if necessary we see that the bound actually
extends to hold for $\sigma \in [\tfrac{1}{5},1)$ too. The proof is complete. 
\end{proof}

\subsection{Iteration and conclusion}

With the excess decay result of Proposition \ref{edprop} at hand we can conclude in a standard manner. 
The first step is obtained by an iteration argument and is in terms of the normalized excess:
$$
\Phi (x_{0},r) = \frac{\mathscr{E}(x_{0},r)}{\mathscr{L}^{n}(B_{r}(x_{0}))}=
\dashint_{B_{r}(x_{0})} \! E \bigl( Du-(Du)_{B_{r}(x_{0})} \bigr) .
$$

\begin{proposition}\label{ite}
Suppose $F \colon \M \to \R$ satisfies $\mathrm{(H0)}$, $\mathrm{(H1)}$, $\mathrm{(H2)}$ and that $u \in \BV (\Omega , \RN )$
is a minimizer. Let $\alpha \in (0,1)$ and $m>0$. Then there exist positive constants $c=c(n,N,\tfrac{L}{\ell},m)$ and 
$\varepsilon = \varepsilon (n,N,\tfrac{L}{\ell},m,\alpha )$ with the following property. If a ball $B_{R}(x_{0}) \subset \Omega$
satisfies
\begin{equation}\label{cond1}
|(Du)_{B_{R}(x_{0})}| < m
\end{equation}
and
\begin{equation}\label{cond2}
\Phi (x_{0},R) < \varepsilon ,
\end{equation}
then
\begin{equation}\label{concl}
\Phi (x_{0},r) \leq c\left( \frac{r}{R} \right)^{2\alpha}\Phi (x_{0},R)
\end{equation}
for all $r \in (0,R)$.
\end{proposition}

\begin{proof}
For ease of notation we write $B_{r}=B_{r}(x_{0})$ and $\Phi (r) = \Phi (x_{0},r)$. 
First recall from Lemma \ref{estE2} that for $\Phi (r) \leq 1$ we have
\begin{equation}\label{pf1}
\dashint_{B_{r}} \! |Du-(Du)_{B_{r}}| \leq \sqrt{3\Phi (r)}
\end{equation}
Consequently, if for a ball $B_{r} \subset \Omega$ we have $|(Du)_{B_{r}}| < m$ and 
$\Phi (r) \leq \tfrac{1}{3}$, then Proposition \ref{edprop} yields
$$
\Phi (\sigma r) \leq c \left( \sigma^{2}+\sigma^{-n} \Phi (r)^{q-1} \right) \Phi (r)
$$
for $q \in (1,\frac{n}{n-1})$, $c=c(n,N,\tfrac{L}{\ell},m,q)$ and $\sigma \in (0,1)$. Fix $q \in (1,\frac{n}{n-1})$
and denote 
\begin{equation}\label{pf2}
C = c(n,N,\tfrac{L}{\ell},m+1,q)
\end{equation}
where we emphasize that we take the constant corresponding to $m+1$ rather than to $m$. With this choice we then
select $\sigma \in (0,1)$ satisfying $C\sigma^{2} < \tfrac{1}{2}\sigma^{2\alpha}$.  For definiteness we fix
\begin{equation}\label{pf3}
\sigma = (3C)^{-\frac{1}{2(1-\alpha )}}.
\end{equation}
Next, take an $\varepsilon_{0} \in (0,\tfrac{1}{3})$ so $C\sigma^{-n}\varepsilon_{0}^{q-1} < \tfrac{1}{2}\sigma^{2\alpha}$, say
\begin{equation}\label{pf4}
\varepsilon_{0} = \left( \frac{\sigma^{n+2\alpha}}{3C} \right)^{\frac{1}{q-1}}.
\end{equation}
Observe that with these choices we have for any ball  $B_{r} \subset \Omega$ satisfying $|(Du)_{B_{r}}| < m+1$
and $\Phi (r) < \varepsilon_0$ that
\begin{equation}\label{pf5}
\Phi (\sigma r) \leq \sigma^{2\alpha}\Phi (r) .
\end{equation}
We iterate this as follows. Let $\varepsilon \in (0, \varepsilon_{0}]$, further restrictions will be imposed below. For the remainder
of the proof we fix a ball $B_{R}=B_{R}(x_{0}) \subset \Omega$
satisfying (\ref{cond1})--(\ref{cond2}). We then have in particular that $\Phi (\sigma R) \leq \sigma^{2\alpha}\varepsilon \leq \varepsilon_0$.
Also, in a standard manner we can estimate
\begin{eqnarray*}
|(Du)_{B_{\sigma R}}| &\leq& |(Du)_{B_{R}}| + |(Du)_{B_{\sigma R}}-(Du)_{B_{R}}|\\
&<& m + \dashint_{B_{\sigma R}} \! |Du - (Du)_{B_{R}}|\\
&\leq& m+\sigma^{-n}\dashint_{B_{R}} \! |Du - (Du)_{B_{R}}|\\
&\stackrel{(\ref{smallE})}{\leq}& m+ \sigma^{-n}\sqrt{3\varepsilon}.
\end{eqnarray*}
We require that $\sigma^{-n}\sqrt{3\varepsilon} \leq 1$, that is,
\begin{equation}\label{pf6}
\varepsilon \leq \frac{\sigma^{2n}}{3}.
\end{equation}
Thus in view of (\ref{pf5}) we have shown that
\begin{equation}\label{pf7}
\Phi (\sigma^{j}R) \leq \sigma^{2\alpha j}\Phi (R) 
\end{equation}
holds for $j=1$, $2$. Let $k \in \N$ and suppose that (\ref{pf7}) holds for $j \in \{ 1, \, \dots \, , \, k \}$. Then
$\Phi (\sigma^{j}R) \leq \sigma^{2\alpha j} \Phi (R) < \sigma^{2\alpha j}\varepsilon < \varepsilon_0$ for each $j \leq k$ and
as above we estimate
\begin{eqnarray*}
|(Du)_{B_{\sigma R}}| &\leq& m+ \sum_{j=1}^{k} \sigma^{-n}\sqrt{3\Phi (\sigma^{j-1}R)}\\
&\leq& m+ \sum_{j=1}^{k} \sigma^{-n}\sqrt{3\sigma^{2\alpha (j-1)}\varepsilon}\\
&<& m+\frac{\sqrt{3\varepsilon}}{\sigma^{n}}\frac{1}{1-\sigma^{\alpha}}.
\end{eqnarray*}
We require that $\frac{\sqrt{3\varepsilon}}{\sigma^{n}}\frac{1}{1-\sigma^{\alpha}} \leq 1$. This is acheived if we take
\begin{equation}\label{pf8}
\varepsilon = \min \{ \varepsilon_{0},\frac{(\sigma^{n}-\sigma^{n+\alpha})^{2}}{3} \} .
\end{equation}
Thus with these choices we have for balls $B_{R}(x_{0}) \subset \Omega$ that satisfy (\ref{cond1})--(\ref{cond2}) shown 
that (\ref{pf7}) holds for all $j \in \N$.
The conclusion follows in a standard manner from this. 
\end{proof}

Using the excess decay estimate of Proposition \ref{ite} we conclude in a routine way with the following $\varepsilon$-regularity
result that in view of Lebesgue's differentiation theorem also implies the last part of Theorem \ref{thm:main}.

\begin{theorem}\label{datheo}
Suppose $F \colon \M \to \R$ satisfies $\mathrm{(H0)}$, $\mathrm{(H1)}$, $\mathrm{(H2)}$ and that $u \in \BV (\Omega , \RN )$
is a minimizer. Then for each $m>0$ there exists $\varepsilon_{m}=\varepsilon_{m} (F) \in (0,1]$ with the following property. If
the ball $B_{R}(x_{0}) \subset \Omega$ satisfies
\begin{equation}\label{1}
|(Du)_{B_{R}(x_{0})}| < m
\end{equation}
and
\begin{equation}\label{2}
\Phi (x_{0},R) < \varepsilon_m ,
\end{equation}
then $u$ is $\CC^{2,\alpha}_{\mathrm{loc}}$ on $B_{\frac{R}{2}}(x_{0})$ for each $\alpha < 1$, and
\begin{equation}\label{3}
\sup_{\stackrel{x,y \in B_{R/4}(x_{0})}{x \neq y}} \frac{| \nabla^{2}u (x)-\nabla^{2}u(y)|^{2}}{|x-y|^{2\alpha}} \leq c\frac{\Phi (x_{0},R)}{R^{2+2\alpha}} 
\end{equation}
where $c=c(n,N,\tfrac{L}{\ell},m,\alpha )$ is a constant.
\end{theorem}

\begin{proof}
We merely sketch the proof as it is essentially standard once the excess decay estimate from Proposition \ref{ite}
has been established. Fix $m>0$ and consider the corresponding 
$$
\tilde{\varepsilon} = \varepsilon (n,N,\tfrac{L}{\ell},m+1,\tfrac{1}{2})>0
$$ 
that was determined in Proposition \ref{ite}. Note that we take the number that corresponds to $m+1$ rather than to $m$. Let
$\varepsilon \in (0,\tilde{\varepsilon}]$ and assume that $B_{R}(x_{0}) \subset \Omega$ is a ball so that
(\ref{cond1})--(\ref{cond2}) hold. We shall determine $\varepsilon$ in the course of the proof. 
Let $x \in B_{R/2}(x_{0})$ and note that the ball $B_{R/2}(x) \subset B_{R}(x_{0})$ satisfies
$$
\Phi (x,\tfrac{R}{2}) \stackrel{\text{Lemma }\ref{estE1}}{\leq} 4 \cdot 2^{n} \Phi (x_{0},R) < 2^{n+2}\varepsilon 
$$
and, proceeding as above,
\begin{eqnarray*}
|(Du)_{B_{\frac{R}{2}}(x)}| &<& m+2^{n}\dashint_{B_{R}(x_{0})} \! | Du-(Du)_{B_{R}(x_{0})}|\\
& \stackrel{(\ref{smallE})}{\leq} & m+2^{n}\sqrt{3\Phi (x_{0},R)}\\
&<& m+2^{n}\sqrt{3\varepsilon}.
\end{eqnarray*}
Thus if we take $\varepsilon = \min \{ \tfrac{\tilde{\varepsilon}}{2^{n+2}}, \tfrac{1}{3 \cdot 2^{2n}} \}$, then Proposition
\ref{ite} yields the bound
$$
\Phi (x,r) \leq c\frac{r}{R} \Phi (x,\tfrac{R}{2}) \leq c_{1}\frac{\Phi (x_{0},R)}{R}r \quad \mbox{ with } \quad c_{1} = 2^{n+2}c
$$
valid for all $x \in B_{R/2}(x_{0})$ and all $r \in (0,\tfrac{R}{2})$. In view of Lemma \ref{estE2} we can deduce a more familiar looking
excess decay estimate:
\begin{eqnarray*}
\left( \dashint_{B_{r}(x)} \! |Du - (Du)_{B_{r}(x)}| \right)^{2} &\leq& \Phi (x,r)^{2}+2\Phi (x,r)\\
&\leq& c_{1}^{2}\left( \tfrac{r}{R} \right)^{2}\Phi (x_{0},R)^{2}+2c_{1}\tfrac{r}{R}\Phi (x_{0},R)\\
&\leq& c\frac{\Phi (x_{0},R)}{R}r 
\end{eqnarray*}
for all $x \in B_{R/2}(x_{0})$ and $r \in (0,R/2)$. (Here $c=c_{1}^{2}+2c_{1}$ and we used that $\Phi (x_{0},R) \leq 1$.)
Using the Campanato-Meyers integral characterization of H\"{o}lder continuity we conclude that $u$ is $\CC^{1,\frac{1}{2}}$ on $B_{R/2}(x_{0})$
and that we have
$$
\sup_{\stackrel{x,y \in B_{R/2}(x_{0})}{x \neq y}} \frac{| \nabla u (x)-\nabla u(y)|^{2}}{|x-y|} \leq c\frac{\Phi (x_{0},R)}{R}
$$
for some constant $c=c(n,N,\tfrac{L}{\ell},m)$. Finally, in order to boost the regularity of $u$ we employ the difference-quotient method
and elliptic Schauder estimates for linear Legendre-Hadamard elliptic systems. Put $B=B_{R/2}(x_{0})$, let $\delta > 0$ be small and denote for 
increments $h \in \R$ with $|h| < \delta R$ the finite difference of $\nabla u$ in
the $j$-th coordinate direction by $\Delta_{j,h}\nabla u (x) = \nabla u(x+he_j )-\nabla u(x)$, $x \in B^{\prime} := B_{(1-\delta )R/2}(x_{0})$. 
Define the $x$-dependent symmetric bilinear forms (for $x \in B^{\prime}$, $|h| < \delta R$ and $1 \leq j \leq n$) by
$$
Q(x)[z,w] = Q_{j,h}(x)[z,w] = \int_{0}^{1} \! F^{\prime \prime}(\nabla u(x)+t\Delta_{j,h}\nabla u(x))[z,w] \, \dd t \quad (z,  \, w \in \M )
$$
From (H0) and the above follows that $Q \in \CC^{0,\tfrac{1}{2}}(B^{\prime}, \bigodot^{2}( \M ))$ with the corresponding Schauder norm of $Q$ 
bounded uniformly in $|h| < \delta R$ and $1 \leq j \leq n$. By virtue of Lemma \ref{propshift} the form $Q$ is uniformly strongly 
Legendre-Hadamard elliptic: there exists a positive constant $c=c(n,N,\tfrac{L}{\ell},m,\mathrm{diam }\Omega )$ such that for all 
$x \in B^{\prime}$ and $a \in \RN$, $b \in \Rn$,
$$
Q(x)[a \otimes b,a \otimes b] \geq \tfrac{1}{c}|a|^{2}|b|^{2} \quad \mbox{ and } \quad |Q(x)| \leq c
$$
hold. Freezing coefficients and using a partition of unity we establish the following G{\aa}rding inequality ($\alpha$, $\beta > 0$)
$$
\int_{B^{\prime}} \! Q(x)[\nabla \varphi,\nabla \varphi] \, \dd x \geq \int_{(B^{\prime}} \! \bigl( \alpha |\nabla \varphi |^{2}-\beta | \varphi |^{2} \bigr) \, \dd x
$$
valid for all $\varphi \in \WW^{1,\infty}_{0} (B^{\prime}, \RN )$, $|h| < \delta R$, $1 \leq j \leq n$. Using these bounds for the form $Q$ 
and testing the Euler-Lagrange system by $\varphi = \Delta_{j,-h}\bigl( \rho^{2}\Delta_{j,h}u \bigr)$ for a suitable cut-off function $\rho$ we 
find in a standard manner that $u \in \WW^{2,2}_{\mathrm{loc}}(B , \RN )$ and that for each direction $1 \leq j \leq n$,
\begin{equation}\label{linearized}
\int_{B} \! F^{\prime \prime}(\nabla u)[\nabla D_{j}u, \nabla \varphi ] \, \dd x = 0 \quad \forall \varphi \in \CC^{1}_{c}(B , \RN )
\end{equation}
It follows by Schauder estimates, see \cite[Theorem 3.2]{Giaquinta}, that $D_j u$ is $\CC^{1,1/2}_{\mathrm{loc}}$ on $B$, and hence
that $u$ is $\CC^{2,1/2}_{\mathrm{loc}}$ on $B$. But then the coefficients $F^{\prime \prime}(\nabla u)$ in the linear elliptic system (\ref{linearized})
are locally Lipschitz and the desired regularity and bound (\ref{3}) follow using Schauder estimates again (see \cite[Theorem 3.3]{Giaquinta}). 
The proof is complete.
\end{proof}

\section{Extensions}

Let $F \colon \M \to \R$ be an integrand of linear growth (\ref{1grow}) which is mean coercive (\ref{mean}), but possibly non-quasiconvex. 
Then for $v \in \BV ( \Omega , \RN )$ and a Lipschitz subdomain $O \subset \Omega$ we define as in (\ref{intro3}) the relaxation from $\WW^{1,1}$:
$$
\F [v,O] = \inf\left\{\liminf_{j \to \infty} \int_{O} \! F(\nabla v_{j}) \, \dd x\colon\; (v_{j})\subset \WW^{1,1}_{v} (O , \RN ),
\;v_{j}\to v\; \text{in} \; \LL^{1}(O ,\RN) \right\} .
$$
The integral representation (\ref{intro4}) remains valid provided we replace $F$ by its quasiconvex envelope $F^{\mathrm{qc}}$, see \cite{KR1}.

In \cite{Almgren2} \textsc{Almgren} extended the elliptic regularity theory in the parametric context for minimizers to also cover various
classes of \emph{almost minimizers}. This allowed him to treat also variational problems with constraints. In the nonparametric context
of quasiconvex variational integrals of $p$-growth for $p>1$ this has been done by \textsc{Duzaar, Grotowski \& Kronz} in \cite{DGK}.
Here we extend Theorem \ref{thm:main} to almost minimizers in the $\BV$ case of linear growth and at the same time localize the result in 
the spirit of \textsc{Acerbi \& Fusco} \cite{AcerbiFusco2} (and \cite{AnGi} in the convex case).

For an increasing continuous function $\omega \colon [0,\infty ) \to \R$ with $\omega (0)=0$ we say that $u \in \BV ( \Omega , \RN )$
is a $\omega$-almost minimizer for $\F$ provided for each ball $B_{r}(x_{0}) \subset \Omega$ we have
\begin{equation}\label{almost}
\F [u,B_{r}(x_{0})] \leq \F [v,B_{r}(x_{0})]+\omega (r)\int_{B_{r}(x_{0})} \! \biggl( |Dv| + \Leb \biggr)
\end{equation}
whenever $v \in \BV ( \Omega , \RN )$ and $u-v$ is supported in $B_{r}(x_{0})$. 

\begin{theorem}\label{general1}
Let $F \colon \M \to \R$ be globally Lipschitz and mean coercive (\ref{mean}). Suppose $u \in \BV ( \Omega , \RN )$ satisfies (\ref{almost}) for 
some function $\omega$ verifying $\limsup_{r \searrow 0} \omega (r)/r^{2\alpha} < \infty$, where $\alpha \in (0,1)$.
Let $z_0 \in \M$ and assume that
$$
\dashint_{B_{r}(x_0 )} \! E(Du-z_{0}\Leb ) \to 0 \mbox{ as } r \searrow 0.
$$
If $F$ is $\CC^{2,1}$ near $z_0$ and if for some $\ell > 0$ the integrand $z \mapsto F(z)-\ell E(z)$ is quasiconvex at $z_0$,
then $u$ is $\CC^{1,\alpha}$ near $x_0$.
\end{theorem}

\noindent
We are not giving the detailed proof for Theorem \ref{general1} here since it follows closely the proof from Section \ref{sec:main}
of Theorem \ref{thm:main}. In order to execute the modified proof one requires the following observation that is closely related
to \cite[Lemma 2.2]{AcerbiFusco2}:

\begin{lemma}\label{modify}
Let $F \colon \M \to \R$ be globally Lipschitz and mean coercive (\ref{mean}), and fix $z_{0} \in \M$. If $F$ is $\CC^2$ near $z_0$
and for some $\ell > 0$ the integrand $z \mapsto F(z)-\ell E(z)$ is quasiconvex at $z_0$, then the quasiconvex envelope $F^{\mathrm{qc}}$
of $F$ is real-valued, satisfies (\ref{mean}), $\mathrm{lip}(F^{\mathrm{qc}})=\mathrm{lip}(F)$, $z \mapsto F^{\mathrm{qc}}(z)-\ell E(z)$
is quasiconvex at $z_0$ and $F^{\mathrm{qc}}=F$ near $z_0$.
\end{lemma}

\begin{proof}
Since $F \geq F^{\mathrm{qc}} \geq (F-\ell E)^{\mathrm{qc}} + \ell E$ and equality holds at $z_0$ we infer that $F^{\mathrm{qc}}-\ell E$
is quasiconvex at $z_0$. In particular, $F^{\mathrm{qc}}$ is then a real-valued quasiconvex integrand. From \cite[Lemma 3.1]{CK}
we deduce that $F^{\mathrm{qc}}$ satisfies (\ref{mean}) with the same constants as $F$. That $\mathrm{lip}(F^{\mathrm{qc}})=\mathrm{lip}(F)$
is a consequence of \cite[Lemma 5.1, Corollary 5.2]{Mat}. Finally, if$F$ is $\CC^2$ on the ball $B_{r}(z_0 )$ and we assume, as we may, that
$F(z_0 ) =0$, $F^{\prime}(z_0 ) =0$, then
\begin{equation}\label{goodbound}
|F(z)| \leq c \Theta \bigl( |z-z_{0}| \bigr) E\bigl( |z-z_{0}| \bigr) \quad \forall z \in \M
\end{equation}
for some constant $c$ and modulus of continuity $\Theta$. We can arrange that $\Theta \colon [0, \infty ) \to [0,1]$ is continuous,
increasing, concave and $\Theta (0)=0$, $\Theta (1)=1$. The proof of (\ref{goodbound}) is implicit in the proof of Lemma 2.2 in 
\cite{AcerbiFusco2} that we may also follow to conclude that $F^{\mathrm{qc}}=F$ on $B_{r/2}(z_{0})$.
\end{proof}

\noindent
As we have dealt with the case of autonomous integrands in the main part of this paper, let us finish by briefly addressing 
the case of $x$-dependent integrands and explain how these can be handled. From a technical perspective, 
the way in which functions are applied to vectorial Radon measures is equally covered by Section~\ref{sec:functionsofmeasures}.
We focus here on a special case and merely state a result that can be made to follow from Theorem \ref{general1}.

\begin{corollary}\label{general2}
Let $F \colon \Omega \times \M \to \R$ be continuous and assume that for some constants $\ell$, $L > 0$ and $\alpha \in (0,1)$ we have
for $x$, $x_1$, $x_2 \in \Omega$ and $z \in \M$,
$$
\left\{
\begin{array}{l}
\ell |z| \leq F(x,z) \leq L(|z|+1),\\
|F(x_{1},z)-F(x_{2},z)| \leq L \min \{ 1,|x_{1}-x_{2}|^{2\alpha} \} (|z|+1),\\
z \mapsto F(x,z) \mbox{ is } \CC^3 \mbox{ and } \partial^{3}F(x,z)/\partial z^{3} \mbox{ is jointly continuous in } (x,z)\\
z \mapsto F(x,z)-\ell E(z) \mbox{ is quasiconvex.} 
\end{array}
\right.
$$
Suppose that $u \in \BV (\Omega , \RN )$ is a minimizer in the sense that 
$$
\int_{\Omega} \! F(x,Du) \leq \int_{\Omega} \! F(x,Dv)
$$
holds for all $v \in \BV ( \Omega , \RN )$ for which $v-u$ has compact support in $\Omega$. 
Then there exists an open subset $\Omega_{u} \subset \Omega$ such that $\Leb ( \Omega \setminus \Omega_{u})=0$ and 
$u$ is $\CC^{1,\alpha}_{\mathrm{loc}}$ on $\Omega_u$.
\end{corollary} 

\noindent
Finally we remark that all the above stated regularity results would extend if instead of the integrand $F=F(Du)$ (or $F=F(x,Du)$) we
considered the integrand $F(Du)+f(x,u)$, where $f \colon \Omega \times \RN \to \R$ is Carath\'{e}odory and satisfies the growth condition
$$
0 \leq f(x,y) \leq c\bigl( |y|^{\tfrac{n}{n-1}} +1 \bigr) \quad \forall (x,y) \in \Omega \times \RN ,
$$
where $c>0$ is a constant (see \cite{Hamburger} for general results in this spirit in the $p$-growth context).
We could also cover the more general notions of almost minimizers considered in \cite{Schmidt1} for the purpose of treating
some image restoration problems.


\bigskip

\noindent
Mathematisches Institut der Univ.~Bonn\\
Endenicher Allee 60, 53111 Bonn\\
Germany
\bigskip

\noindent
Mathematical Institute, University of Oxford, Andrew Wiles Building\\ 
Radcliffe Observatory Quarters, Woodstock Road, Oxford OX2 6GG\\ 
United Kingdom

\end{document}